\documentclass[a4paper,11pt]{scrartcl}

\usepackage[utf8]{inputenc} 
\usepackage[T1]{fontenc}
\usepackage{graphicx} 

\usepackage{caption} 

\usepackage{float}
\usepackage{mathpazo}
\usepackage{booktabs} 
\usepackage{array} 
\usepackage{paralist} 
\usepackage{verbatim} 
\usepackage{subfig}
\usepackage{enumitem}
\usepackage{amsmath} 
\usepackage{amssymb}
\usepackage{amsthm}
\usepackage{verbatim}
\usepackage{hyperref}
\usepackage{extarrows}
\usepackage{url}
\usepackage{listings}

\usepackage[a4paper,lmargin={3cm},rmargin={2cm},
tmargin={2.5cm},bmargin = {2.5cm}]{geometry}

\title{The Proximal Alternating Minimization Algorithm for two-block separable convex optimization problems with linear constraints}
\author{Sandy Bitterlich\footnote{Chemnitz University of Technology, Faculty  of Mathematics, Reichenhainer Straße 39, 09126 Chemnitz, Germany, 
email: \href{mailto:sandy.bitterlich@mathematik.tu-chemnitz.de}{sandy.bitterlich@mathematik.tu-chemnitz.de}. Research supported by DFG, project WA922/9-1.} \quad Radu Ioan Bo\c{t} \footnote{University of Vienna, Faculty  of Mathematics, Oskar-Morgenstern-Platz 1, 
A-1090 Vienna, Austria, email: \href{mailto:radu.bot@univie.ac.at}{radu.bot@univie.ac.at}. Research partially supported by FWF, project I 2419-N32.} 
\quad Ern\"{o} Robert Csetnek \footnote{University of Vienna, Faculty  of Mathematics, Oskar-Morgenstern-Platz 1, A-1090 Vienna, 
Austria, email: \href{mailto:ernoe.robert.csetnek@univie.ac.at}{ernoe.robert.csetnek@univie.ac.at}. Research supported by FWF, project P 29809-N32.} 
\\ Gert Wanka \footnote{Chemnitz University of Technology, Faculty  of Mathematics, Reichenhainer Straße 39, 09126 Chemnitz, Germany, 
email: \href{mailto:gert.wanka@mathematik.tu-chemnitz.de}{gert.wanka@mathematik.tu-chemnitz.de}. Research partially supported by DFG, project WA922/9-1.}}
\newcommand{\R}{\mathbb{R}}
\newcommand{\HH}{\mathcal{H}}
\newcommand{\GG}{\mathcal{G}}
\newcommand{\h}{\ensuremath{\mathcal{H}}}
\newcommand{\Z}{\ensuremath{\mathcal{Z}}}

\DeclareMathOperator{\ri}{ri}
\DeclareMathOperator{\sqri}{sqri}
\DeclareMathOperator{\inte}{int}
\DeclareMathOperator{\dom}{dom}
\DeclareMathOperator*{\argmin}{argmin}

\DeclareMathOperator{\Id}{Id}
\DeclareMathOperator*\prox{Prox}%

\newtheorem{theorem}{Theorem}

\newtheorem{algorithm}[theorem]{Algorithm}
\newtheorem{remark}[theorem]{Remark}
\newtheorem{problem}[theorem]{Problem}

\begin{document}
\maketitle  



\pagestyle{headings}

\noindent \textbf{Abstract.} The Alternating Minimization Algorithm (AMA) has been proposed by Tseng to solve convex programming problems 
with two-block separable linear constraints and objectives, whereby (at least) one of the components of the latter is assumed to be strongly convex. The fact that one of the subproblems to be solved within the 
iteration process of AMA does not usually correspond to the calculation of a proximal operator through a closed formula, affects the implementability of the algorithm. In this paper we allow in 
each block of the objective a further smooth convex function and propose a  proximal version of AMA, called Proximal AMA, which is achieved by equipping the algorithm with proximal terms induced by variable metrics. 
For suitable choices of the latter, the solving of the two subproblems in the iterative scheme can be reduced to the computation of proximal operators. We investigate the convergence of the proposed algorithm in a real Hilbert space setting and
illustrate its numerical performances on two applications in image processing and machine learning.\vspace{1ex}

\noindent \textbf{Key Words.} Proximal AMA, Lagrangian, saddle points, subdifferential, convex optimization, Fenchel duality \vspace{1ex}

\noindent \textbf{AMS subject classification.} 47H05, 65K05, 90C25

\section{Introduction and preliminaries}\label{sec1}

The Alternating Minimization Algorithm (AMA) has been proposed by Tseng (see \cite{tseng91}) in order to solve optimization problems of the form
\begin{align}\label{eq:Problem_AMA}
&\inf_{x \in R^n,z \in \R^m} f(x)+g(z),\\
&\quad~~~\text{s.t. } \quad Ax +Bz=b. \nonumber
\end{align}
where $f:\R^n \to \overline{\R}:=\R\cup\{\pm\infty\}$ is  a proper, $\gamma$-strongly convex with $\gamma>0$ (this means that
$f-\frac{\gamma}{2}\|\cdot\|^2$ is convex) and lower semicontinuous function, $g :\R^m \to \overline{\R}$ is a proper, 
convex and lower semicontinuous function, $A \in \mathbb{R}^{r \times n}, B \in \mathbb{R}^{r \times m}$ and $b \in \R^r$. 

For $c > 0$ we consider the augmented Lagrangian associated with problem \eqref{eq:Problem_AMA}
\begin{equation*}
L_c:\R^n \times \R^m \times \R^r \to \overline{\R}, \quad L_c(x,z,p)=f(x)+g(z)+\langle p, b-Ax-Bz\rangle + \frac{c}{2}\|Ax+Bz-b\|^2.
\end{equation*}
The Lagrangian associated with problem \eqref{eq:Problem_AMA} is
\begin{align*}
L :\R^n \times \R^m \times \R^r \to \overline{\R}, \quad L (x,z,p) =  f(x)+g(z)+\langle p, b-Ax-Bz\rangle.
\end{align*}

The Alternating Minimization Algorithm reads:
\begin{algorithm}\label{alg2}(AMA) Choose $p^0 \in \R^r$ and a sequence of stepsizes $(c_k)_{k\geq 0} \subseteq (0,+\infty)$. For all $k \geq 0$ set:
	\begin{align}
	x^k&= \argmin_{x \in \R^n}\left\{f(x)-\langle p^k,Ax\rangle \right\}\label{eq:AMA_update_x}\\
	z^k&\in\argmin_{z \in \R^m}\left\{g(z)-\langle p^k,Bz\rangle+\frac{c_k}{2} \|Ax^k+Bz-b\|^2\right\}\label{eq:AMA_update_z}\\
	p^{k+1}&=p^k+c_k(b-Ax^k-Bz^k).
	\end{align}
\end{algorithm}

The main convergence properties of this numerical algorithm are summarized in the theorem below (see \cite{tseng91}). 

\begin{theorem}
Let $A \neq 0$ and $(x,z) \in \ri(\dom f) \times \ri(\dom g)$ be such that $Ax+Bz=b$.  Assume that the sequence of stepsizes $(c_k)_{k \geq 0}$ satisfies
$$\epsilon \leq c_k \leq \frac{2 \gamma}{\|A\|^2}-\epsilon \ \forall k \geq 0,$$ 
where $\epsilon \in \left (0, \frac{\gamma}{\|A\|^2} \right)$. Let $(x^k, z^k, p^k)_{k \geq 0}$ be the sequence generated by Algorithm \ref{alg2}. Then there exist $x^* \in \R^n$ and an optimal Lagrange multiplier $p^* \in \R^r$ associated
with the constraint $Ax+Bz=b$ such that
\begin{align*}
x^k \to x^*, Bz^k \to b - Ax^*, p^k \rightarrow p^* (k \rightarrow +\infty).
		\end{align*}
If the function $z \mapsto g(z)+\|Bz\|^2$ has bounded level sets, then  $(z^k)_{k \geq 0}$ is bounded and any of its cluster points $z^*$ provides with $(x^*,z^*)$ an optimal solution of \eqref{eq:Problem_AMA}.
\end{theorem}

The strong convexity of $f$ allows to reduce the minimization problem in \eqref{eq:AMA_update_x} to the calculation of the proximal operator of a proper, convex and lower semicontinuous function. This is
for the minimization problem in \eqref{eq:AMA_update_z}, due to the presence of the linear operator $B$, in general not the case. This fact makes the AMA method not very tractable for implementation issues. With the exception of some very particular cases, one has to use 
a subroutine in order to compute $(z^k)_{k \geq 0}$, a fact which can have a  negative influence on the convergence behaviour of the algorithm. One possibility to avoid this, without
losing the convergence properties of AMA, is to replace \eqref{eq:AMA_update_z} by a proximal step of $g$.  The papers \cite{bete09} and \cite{cope11} provide convincing evidences for the versatility and efficiency of proximal point algorithms for solving
nonsmooth convex optimization problems.

In this paper we address in a real Hilbert space setting a problem of type \eqref{eq:Problem_AMA}, which is obtained by adding in each block of the objective a further smooth convex function. To solve this problem we propose a 
so-called Proximal Alternating Minimization Algorithm (Proximal AMA), which is obtained by inducing in each of the minimization  problems \eqref{eq:AMA_update_x} and \eqref{eq:AMA_update_z} additional proximal terms defined by means of positively semidefinite operators. The two smooth
convex functions in the objective are evaluated via gradient steps. We will show that, for appropriate choices of these operators, the minimization problem in \eqref{eq:AMA_update_z} reduces to the performing of a proximal step.  We perform the convergence analysis 
of the proposed method and show that the generated sequence converges weakly to a saddle point of the Lagrangian associated with the optimization problem under investigation.The numerical performances of Proximal AMA, in particular in comparison with AMA, 
are illustrated on two applications in image processing and machine learning.

A similarity of AMA to the classical ADMM algorithm, introduced by Gabay and Mercier in \cite{game76}, is evident. In \cite{fpst, ST} (see also \cite{babo17, bocs17}) proximal versions of the ADMM algorithm have been proposed and investigated from the point of view of their convergence properties.
Parts of the convergence analysis for Proximal AMA are carried out in a similar spirit to the convergence proofs in these papers. 

In the remainder of this section, we discuss some notations, definitions and basic properties we will use in this paper (see \cite{baco17}). Let $\mathcal{H}$ and $\GG$ be real Hilbert spaces with corresponding inner products $\langle \cdot, \cdot \rangle$ and associated norms 
$\|\cdot\|=\sqrt{\langle \cdot, \cdot \rangle}$. In both spaces we denote by $\rightharpoonup$ the weak convergence and by $\rightarrow$ the strong convergence.

We say that a function $f:\mathcal{H} \to \overline{\mathbb{R}}$  is proper, if $\dom f:=\{x\in \mathcal{H}: f(x)<+\infty\}\neq\emptyset$ and $f(x) > -\infty$ for all $x \in \mathcal{H}$. Let be
\begin{equation*}
	\Gamma(\mathcal{H})=\{f:\mathcal{H} \to \overline{\mathbb{R}}: f \text{ is proper, convex and lower semicontinuous} \}.
\end{equation*}
Let be $f\in \Gamma(\mathcal{H})$. The (Fenchel) conjugate function $f^*:\mathcal{H} \to \overline{\mathbb{R}}$ of $f$ is defined as
\begin{equation*}
f^*(p)=\text{sup}_{x \in \HH}\{\langle p, x \rangle -f(x)  \}  ~ \forall p \in \HH
\end{equation*}
and is a proper, convex and lower semicontinuous function. It also holds $f^{**}=f$, where $f^{**}$ is the conjugate function of $f^*$.
The (convex) subdifferential of $f$ is defines as $\partial f(x)=\{u\in \mathcal{H}: f(y)\geq f(x)+\langle u,y-x\rangle \forall y\in \mathcal{H}\}$, if $f(x) \in \R$,
and as $\partial f(x) = \emptyset$, otherwise.

The infimal convolution of two proper functions $f,g:{\cal H}\rightarrow \overline{\mathbb{R}}$ is the function $f\Box g:{\cal H}\rightarrow\overline{\mathbb{R}}$, 
defined by $(f\Box g)(x)=\inf_{y\in {\cal H}}\{f(y)+g(x-y)\}$. 

The proximal point operator of parameter $\gamma$ of $f$ at $x$, where $\gamma >0$,  is defined as
\begin{equation*}
	\text{Prox}_{\gamma f} : {\cal H} \rightarrow {\cal H}, \quad \text{Prox}_{\gamma f}(x)=\argmin_{y \in \mathcal{H}}\left\{\gamma f(y)+\frac{1}{2}\|y-x\|^2\right\}.
\end{equation*}
According to Moreau's decomposition formula we have
\begin{equation*}
\prox\nolimits_{\gamma f}(x)+\gamma\prox\nolimits_{(1/\gamma)f^*}(\gamma^{-1}x)=x, \  \ \forall x \in \HH.
\end{equation*}

Let $C \subseteq \mathcal{H}$ be a convex and closed set. The strong quasi-relative interior of $C$ is
\begin{equation*}
	\text{sqri}(C)=\left\{x \in C: \cup_{\lambda>0}\lambda(C-x) \text{ is a closed linear subspace of } \mathcal{H}\right\}.
\end{equation*}
We always have $\inte(C)\subseteq \text{sqri}(C)$ and, if $\HH$ is finite dimensional, then $\text{sqri}(C)=\text{ri}(C),$ 
where $\text{ri}(C)$ denotes the relative interior of $C$ and represents the interior of $C$ relative to its affine hull. 

We set
\begin{equation*}
S_+(\HH)=\{M:\HH \to \HH: M \text{ is linear, continuous, self-adjoint and positive semidefinite}\}.
\end{equation*}
For $M \in S_+(\HH)$ we define the seminorm $\|\cdot\|_{M} : \HH \rightarrow [0,+\infty)$,  $\|x\|_M= \sqrt{\langle x, Mx \rangle}$.
We consider the Loewner partial ordering on $S_+(\HH)$, defined for $M_1, M_2 \in \mathcal{S}_+(\mathcal{H})$ by
\[M_1 \succcurlyeq M_2 \Leftrightarrow \|x\|^2_{M_1} \geq \|x\|^2_{M_2} ~ \forall x \in \mathcal{H}.
\]
Furthermore, we define for $\alpha>0$
\begin{equation*}
\mathcal{P}_\alpha(\HH):=\{M\in \mathcal{S}_+(\HH): M \succcurlyeq \alpha \textrm{Id} \},
\end{equation*}
where $\textrm{Id} : \HH \rightarrow \HH, \textrm{Id}(x) = x$ for all $x \in \HH$, denotes the identity operator on $\HH$.

Let $A:\HH \to \GG$ be a linear continuous operator. The operator $A^*:\GG \to \HH$, fulfilling $\langle A^*y,x \rangle = \langle y, Ax \rangle $ for all $x \in \HH$ and $y \in \GG$, denotes the adjoint operator of $A$, while 
$\|A\|:=\sup\{\|Ax\| : \|x\| \leq 1 \}$ denotes the norm of $A$.

\section{The Proximal Alternating Minimization Algorithm}\label{sec2}

The two-block separable optimization problem we are going to investigate has the following formulation.

\begin{problem}\label{probl:prox-AMA-basic_problem}
	Let $\mathcal{H}$, $\GG$ and $\mathcal{K}$ be real Hilbert spaces, $f \in \Gamma(H)$ $\gamma$-strongly convex with $\gamma>0$, 
	$g \in \Gamma(G)$, $h_1:\HH\to \R$ a convex and Fr\'echet differentiable function with $L_1$-Lipschitz continuous gradient, $L_1 \geq 0$, $h_2:\GG \to \R$  a convex and Fr\'{e}chet differentiable functions with $L_2$-Lipschitz continuous gradient, $L_2 \geq 0$, 
	$A:\HH \to \mathcal{K}$ and $B:\GG \to \mathcal{K}$ 
	linear continuous operators such that $A \neq 0$ and $b \in \mathcal{K}$. Consider the following optimization problem with two-block separable objective function and  linear constraints
	\begin{align}\label{opt:Prox-AMA:primal}
	&\min_{x \in \HH, z \in \GG} f(x)+h_1(x)+g(z)+h_2(z).\\
	&~\text{s.t. }Ax+Bz=b \nonumber
	\end{align}
\end{problem}
Notice that we allow the Lipschitz constant of the gradient of the function $h_1$ to be zero. In this case $h_1$ is an affine function. The same applies for the function $h_2$.

The Lagrangian associated with the optimization problem \eqref{opt:Prox-AMA:primal} is
	\begin{align*}
	L :\HH\times \GG \times \mathcal{K} \to \overline{\R}, \quad L(x,z,p)=f(x)+h_1(x)+g(z)+h_2(z)+\langle p, b-Ax-Bz \rangle. 
	\end{align*}

	We say that $(x^*,z^*,p^*) \in \HH \times \GG \times \mathcal{K}$ is a saddle point of the Lagrangian $L$, if
	\[L(x^*,z^*,p) \leq L(x^*,z^*,p^*) \leq L(x,z,p^*)	\]
	holds for all $(x,z,p) \in \HH\times \GG \times \mathcal{K}$.
	
 One can show that $(x^*,z^*,p^*)$ is a saddle point of the Lagrangian $L$ if and only if $(x^*,z^*)$ is an optimal solution of \eqref{opt:Prox-AMA:primal}, 
  $p^*$ is an optimal solution of its Fenchel dual problem
  \begin{equation}\label{Fenchel-dual} \sup_{\lambda\in \mathcal{K}}\{-(f^*\Box h_1^*)(A^*\lambda)-(g^*\Box h_2^*)(B^*\lambda)+\langle \lambda,b\rangle\}, 
  \end{equation}
and the optimal objective values of \eqref{opt:Prox-AMA:primal} and \eqref{Fenchel-dual} coincide. The existence of saddle points for $L$ is guaranteed when \eqref{opt:Prox-AMA:primal} has an optimal solution and, for instance, 
the Attouch-Br\'{e}zis-type condition 
\begin{equation}\label{reg-cond} b\in\sqri (A(\dom f)+B(\dom g))
 \end{equation}
holds (see \cite[Theorem 3.4]{b-hab}). In the finite dimensional setting, this asks for the existence of 
$x \in \ri(\dom f )$ and $z \in \ri(\dom g)$ satisfying $Ax+Bz=b$ and coincides with the assumption used by Tseng in \cite{tseng91}.

The system of optimality conditions for the primal-dual pair of optimization problems \eqref{opt:Prox-AMA:primal}-\eqref{Fenchel-dual} reads:
\begin{equation}\label{opt-cond} A^*p^*-\nabla h_1(x^*) \in \partial f(x^*), \ 
	B^*p^*-\nabla h_2(z^*)\in \partial g(z^*) \ \mbox{ and } Ax^*+Bz^*=b. 
\end{equation}
This means that if \eqref{opt:Prox-AMA:primal} has an optimal solution $(x^*,z^*)$ and a qualification condition, like for instance \eqref{reg-cond}, is fulfilled, 
then there exists an optimal solution $p^*$ of \eqref{Fenchel-dual} such that \eqref{opt-cond} holds, consequently, $(x^*,z^*,p^*)$ is a saddle point of the 
Lagrangian $L$. Conversely, if $(x^*,z^*,p^*)$ is a saddle point of the Lagrangian $L$, thus, $(x^*,z^*,p^*)$ satisfies relation \eqref{opt-cond}, then $(x^*,z^*)$ is an optimal solution of \eqref{opt:Prox-AMA:primal} and
$p^*$ is an optimal solution of \eqref{Fenchel-dual}. 

\begin{remark} If $(x_1^*,z_1^*,p_1^*)$ and $(x_2^*,z_2^*,p_2^*)$ are two saddle points of the Lagrangian $L$, then $x_1^*=x_2^*$. This follows easily by using the strong monotonicity of $\partial f$, the monotonicity of $\partial g$ and the relations in \eqref{opt-cond}.
\end{remark}

In the following we formulate the Proximal Alternating Minimization Algorithm to solve \eqref{opt:Prox-AMA:primal}. To this end, we modify Tseng's AMA by evaluating in each of the two subproblems the functions $h_1$ and $h_2$ via gradient steps, respectively, and by 
introducing proximal terms defined through two sequence of positively semidefinite operators $(M_1^k)_{k \geq 0}$ and $(M_2^k)_{k \geq 0}$. 

\begin{algorithm}(Proximal AMA) \label{alg:prox-AMA-h}
	Let $(M_1^k)_{k \geq 0} \subseteq \mathcal{S}_+(\HH)$ and $(M_2^k)_{k \geq 0} \subseteq \mathcal{S}_+(\GG)$. Choose $(x^0,z^0,p^0) \!\!\in \HH \times \GG \times \mathcal{K}$ and a sequence of stepsizes $(c_k)_{k\geq 0} \subseteq (0,+\infty)$. 
	For all $k\geq 0$ set:
	\begin{align}
		x^{k+1}&=\argmin_{x \in \HH}\left\{f(x)-\langle p^k,Ax\rangle 
		     +\langle x-x^{k}, \nabla h_1(x^{k}) \rangle
		     +\frac{1}{2}\|x-x^{k}\|^2_{M_1^{k}}\right\}\label{eq:prox-AMA-h-x-Update}\\
		z^{k+1}& \in  \argmin_{z \in \GG}\left\{g(z)-\langle p^k,Bz\rangle
			+\frac{c_k}{2} \|Ax^{k+1}+Bz-b\|^2  +\langle z-z^k, \nabla h_2(z^k) \rangle 
			+\frac{1}{2}\|z-z^k\|^2_{M_2^{k}}\right\}  \label{eq:prox-AMA-h-z-Update}\\
		p^{k+1}&=p^k+c_k(b-Ax^{k+1}-Bz^{k+1}).\label{eq:prox-AMA-h-p-Update}
	\end{align}
\end{algorithm}

\begin{remark}\textrm
	The sequence $(z^k)_{k \geq 0}$ is uniquely determined if there exists $\alpha_k > 0$ such 
	that $c_kB^*B+M_2^k \in \mathcal{P}_{\alpha_k}(\GG)$ for all $k \geq 0$. 
	This actually ensures that the objective function in the subproblem \eqref{eq:prox-AMA-h-z-Update} is strongly convex.
\end{remark}

\begin{remark}\label{prox-step}\textrm
Let $k \geq 0$ be fixed and $M_2^k:=\frac{1}{\sigma_k}\text{Id}-c_kB^* B$, where $\sigma_k > 0$ and $\sigma_kc_k\|B\|^2 \leq 1$. Then $M_2^k$ is positively semidefinite and the update of $z^{k+1}$ in the Proximal AMA method becomes a proximal step. Indeed,  \eqref{eq:prox-AMA-h-z-Update} holds if and only if
	\[0 \in \partial g(z^{k+1})+(c_kB^* B + M_2^k)z^{k+1}+c_kB^*(Ax^{k+1}-b)-M_2^kz^{k}+\nabla h_2(z^{k})-B^* p^k\,
	\]
or, equivalently,
	\[0 \in \partial g(z^{k+1})+\frac{1}{\sigma_k}z^{k+1} - \left(\frac{1}{\sigma_k}\Id-c_kB^* B\right)z^{k}+\nabla h_2(z^{k})+
	c_kB^*(Ax^{k+1}-b)-B^* p^k.
	\]
But this is nothing else than
	\begin{align*}
	z^{k+1} &=\argmin_{z \in \GG}\left\{g(z)+\frac{1}{2\sigma_k}\left\|z-\left(z^{k}-\sigma_k\nabla h_2(z^{k})+\sigma_kc_kB^*(b-Ax^{k+1}-Bz^k)+\sigma_kB^* p^k \right)\right\|^2
	\right\}\\
&=\prox\nolimits_{\sigma_k g}\left(z^{k}-\sigma_k\nabla h_2(z^{k})+\sigma_kc_kB^*(b-Ax^{k+1}-Bz^k)+\sigma_kB^* p^k \right).
	\end{align*}
\end{remark}

The convergence of the Proximal AMA method is addressed in the next theorem. 

\begin{theorem}\label{th:Prox-AMA-convergence}
In the setting of Problem \ref{probl:prox-AMA-basic_problem} let the set of the saddle points of the Lagrangian $L$ be nonempty. Assume that $M_1^k-\frac{L_1}{2}\Id \in \mathcal{S}_+(\mathcal{H}), M_1^k \succcurlyeq M_1^{k+1}, M_2^k-\frac{L_2}{2}\Id \in \mathcal{S}_+(\mathcal{G}), M_2^k \succcurlyeq M_2^{k+1}$ for all $k \geq 0$ and that $(c_k)_{k \geq 0}$ is a monotonically decreasing sequence satisfying
	\begin{equation}\label{as:Prox-AMA:c_k}
	\epsilon \leq c_k \leq \frac{2\gamma}{\|A\|^2}-\epsilon \quad \forall k \geq 0,
	\end{equation}
	where $\epsilon \in \left (0, \frac{\gamma}{\|A\|^2}\right)$. If one of the following assumptions: 
	\renewcommand{\labelenumi}{\roman{enumi}}
	\begin{enumerate}
		\item[(i)] there exists $\alpha>0$ such that $M_2^k-\frac{L_2}{2}\Id \in \mathcal{P}_{\alpha}(\mathcal{G})$ for all $k \geq 0$;
		\item[(ii)] there exists $\beta>0$ such that $B^*B\in \mathcal{P}_{\beta}(\mathcal{G})$;
	\end{enumerate}
holds true, then the sequence $(x^k,z^k,p^k)_{k \geq 0}$ generated by Algorithm \ref{alg:prox-AMA-h} converges weakly to a saddle point of the Lagrangian $L$.
\end{theorem}

\begin{proof}
	Let $(x^*,z^*,p^*)$ be a fixed saddle point of the Lagrangian $L$. This means that it fulfils the system of optimality conditions
	\begin{align} 
        & A^*p^*-\nabla h_1(x^*)  \in \partial f(x^*) \label{eq:Prox-AMA-oc1}\\
	& B^*p^*-\nabla h_2(z^*) \in \partial g(z^*) \label{eq:Prox-AMA-oc2}\\
        & Ax^*+Bz^* =b \label{eq:Prox-AMA-const}
	\end{align} 
We start by proving that
	\begin{equation*}
	\sum_{k\geq 0}\|x^{k+1}-x^*\|^2< +\infty, \quad  \sum_{k\geq 0}\|Bz^{k+1}-Bz^*\|^2< +\infty, \quad
	\sum_{k\geq 0}\|z^{k+1}-z^k\|^2_{M_2^{k}-\frac{L_2}{2}\text{Id}}< +\infty
	\end{equation*}
	and that the sequences $(z^k)_{k \geq 0}$ and $(p^k)_{k \geq 0}$ are bounded.
	
Assume that $L_1 >0$ and $L_2 >0$. Let $k \geq 0$ be fixed. Writing the optimality conditions for the subproblems \eqref{eq:prox-AMA-h-x-Update} and \eqref{eq:prox-AMA-h-z-Update} we obtain
	\begin{align}
	A^*p^k-\nabla h_1(x^{k}) + M_1^k(x^{k}-x^{k+1}) & \in \partial f(x^{k+1}) \label{eq:Prox-AMA-ocAlg1}
	\end{align}
and
	\begin{align}
	B^*p^k-\nabla h_2(z^{k}) + c_kB^*(-Ax^{k+1}-Bz^{k+1}+b)+M_2^k(z^{k}-z^{k+1}) &\in \partial g(z^{k+1}), \label{eq:Prox-AMA-ocAlg2}
	\end{align}
respectively.	
Combining \eqref{eq:Prox-AMA-oc1}, \eqref{eq:Prox-AMA-oc2}, \eqref{eq:Prox-AMA-ocAlg1}, \eqref{eq:Prox-AMA-ocAlg2} with the strong monotonicity of $\partial f$ and the monotonicity of $\partial g$, it yields
	\[\langle A^*(p^k-p^*)-\nabla h_1(x^{k})+\nabla h_1(x^*)+M_1^k(x^{k}-x^{k+1}), x^{k+1}-x^* \rangle \geq \gamma \|x^{k+1}-x^*\|^2
	\]
	and
	\[\langle B^*(p^k-p^*)-\nabla h_2(z^{k})+\nabla h_2(z^*)+ c_kB^*(-Ax^{k+1}-Bz^{k+1}+b)+M_2^k(z^{k}-z^{k+1}), z^{k+1}-z^* \rangle \geq 0,
	\]
which after summation lead to
	\begin{align}\label{eq21}
	\langle p^k-p^*,Ax^{k+1}-Ax^*\rangle + \langle p^k-p^*,Bz^{k+1}-Bz^*\rangle & \nonumber\\
+ \langle c_k(-Ax^{k+1}-Bz^{k+1}+b),Bz^{k+1}-Bz^*\rangle & \nonumber \\
- \langle \nabla h_1(x^{k})-\nabla h_1(x^*), x^{k+1}-x^*\rangle
	 -\langle \nabla h_2(z^{k})-\nabla h_2(z^*), z^{k+1}-z^*\rangle& \nonumber \\
	 + \langle M_1^k(x^{k}-x^{k+1}),x^{k+1}-x^*\rangle + \langle M_2^k(z^{k}-z^{k+1}),z^{k+1}-z^*\rangle & \geq \gamma \|x^{k+1}-x^*\|^2.
	\end{align}
According to the Baillon-Haddad-Theorem (see \cite[Corollary 18.16]{baco17}) the gradients of $h_1$ and $h_2$ are $\frac{1}{L_1}$ and $\frac{1}{L_2}$-cocoercive, respectively, thus 
	\begin{align*}
	\langle \nabla h_1(x^*)-\nabla h_1(x^{k}), x^*-x^{k}\rangle\geq \frac{1}{L_1}\|\nabla h_1(x^*)-\nabla h_1(x^{k})\|^2\\
	\langle \nabla h_2(z^*)-\nabla h_2(z^{k}), z^*-z^{k}\rangle\geq \frac{1}{L_2}\|\nabla h_2(z^*)-\nabla h_2(z^{k})\|^2.
	\end{align*}
On the other hand, by taking into account  \eqref{eq:prox-AMA-h-p-Update} and \eqref{eq:Prox-AMA-const}, it holds:
	\begin{align*}
	\langle p^k-p^*,Ax^{k+1}-Ax^*\rangle + \langle p^k-p^*,Bz^{k+1}-Bz^*\rangle&= \langle p^k-p^*, Ax^{k+1}+Bz^{k+1}-b \rangle\\& =  \frac{1}{c_k}\langle p^k-p^*, p^k-p^{k+1}\rangle
	\end{align*}  
By employing the last three relations in \eqref{eq21}, it yields
	\begin{align*}
	\frac{1}{c_k}\langle p^k-p^*, p^k-p^{k+1}\rangle + c_k\langle -Ax^{k+1}-Bz^{k+1}+b, Bz^{k+1}-Bz^* \rangle& \\
	 + \langle M_1^k(x^{k}-x^{k+1}),x^{k+1}-x^*\rangle + \langle M_2^k(z^{k}-z^{k+1}),z^{k+1}-z^*\rangle&\\
	 +\langle \nabla h_1(x^{*})-\nabla h_1(x^{k}), x^{k+1}-x^*\rangle+\langle \nabla h_1(x^{*})-\nabla h_1(x^{k}),x^*- x^{k}\rangle&\\
	 -\frac{1}{L_1}\|\nabla h_1(x^*)-\nabla h_1(x^{k})\|^2
	 +\langle \nabla h_2(z^{*})-\nabla h_2(z^{k}), z^{k+1}-z^*\rangle &\\
	 +\langle \nabla h_2(z^{*})-\nabla h_2(z^{k}), z^*-z^{k}\rangle-\frac{1}{L_2}\|\nabla h_2(z^*)-\nabla h_2(z^{k})\|^2& \geq \gamma\|x^{k+1}-x^*\|^2,
	\end{align*}
which, after expressing the inner products by means of norms, becomes
	\begin{align*}
	\frac{1}{2c_k}\left(\|p^k-p^*\|^2+\|p^k-p^{k+1}\|^2-\|p^{k+1}-p^*\|^2\right)&\\
	+\frac{c_k}{2}\left(\|Ax^*-Ax^{k+1}\|^2-\|b-Ax^{k+1}-Bz^{k+1}\|^2-\|Ax^*+Bz^{k+1}-b\|^2\right)&\\
	+\frac{1}{2}\left(\|x^{k}-x^*\|^2_{M_1^k}-\|x^{k}-x^{k+1}\|^2_{M_1^k}-\|x^{k+1}-x^*\|^2_{M_1^k}\right)&\\
	+\frac{1}{2}\left(\|z^{k}-z^*\|^2_{M_2^k}-\|z^{k}-z^{k+1}\|^2_{M_2^k}-\|z^{k+1}-z^*\|^2_{M_2^k}\right)&\\
	+\langle \nabla h_1(x^{*})-\nabla h_1(x^{k}), x^{k+1}-x^{k}\rangle
	-\frac{1}{L_1}\|\nabla h_1(x^*)-\nabla h_1(x^{k})\|^2 &\\
	+\langle \nabla h_2(z^{*})-\nabla h_2(z^{k}), z^{k+1}-z^{k}\rangle
	-\frac{1}{L_2}\|\nabla h_2(z^*)-\nabla h_2(z^{k})\|^2 & \geq \gamma\|x^{k+1}-x^*\|^2.
	\end{align*}
	Using again \eqref{eq:prox-AMA-h-p-Update}, the inequality $\|Ax^*-Ax^{k+1}\|^2\leq \|A\|^2\|x^*-x^{k+1}\|^2$ and the following expressions 
	\begin{align*}
	\langle \nabla h_1(x^*)-\nabla h_1(x^{k}), x^{k+1}-x^{k}\rangle - \frac{1}{L_1}\|\nabla h_1(x^*)-\nabla h_1(x^{k})\|^2\\
	=-L_1\left \|\frac{1}{L_1}(\nabla h_1(x^*)-\nabla h_1(x^{k}))+\frac{1}{2}(x^{k}-x^{k+1}) \right \|^2+\frac{L_1}{4}\|x^{k}-x^{k+1}\|^2,
	\end{align*}
and
\begin{align*}
	\langle \nabla h_2(x^*)-\nabla h_2(z^{k}), z^{k+1}-z^{k}\rangle - \frac{1}{L_2}\|\nabla h_2(z^*)-\nabla h_2(z^{k})\|^2\\
	=-L_2\left \|\frac{1}{L_2}(\nabla h_2(z^*)-\nabla h_2(z^{k}))+\frac{1}{2}(z^{k}-z^{k+1})\right \|^2+\frac{L_2}{4}\|z^{k}-z^{k+1}\|^2,
	\end{align*}
it yields
	\begin{align*}
	\frac{1}{2c_k}\|p^{k+1}-p^*\|^2+\frac{1}{2}\|z^{k+1}-z^*\|^2_{M_2^k}& \leq \\
	\frac{1}{2c_k}\|p^k-p^*\|^2+\frac{1}{2}\|z^{k}-z^*\|^2_{M_2^k}-\frac{c_k}{2}\|Ax^*+Bz^{k+1}-b\|^2&\\
	-\frac{1}{2}\|z^{k}-z^{k+1}\|^2_{M_2^k}-\left(\gamma-\frac{c_k}{2}\|A\|^2\right)\|x^{k+1}-x^*\|^2-\frac{1}{2}\|x^{k}-x^{k+1}\|^2_{M_1^k}&\\
	-L_1\left\|\frac{1}{L_1}(\nabla h_1(x^*)-\nabla h_1(x^{k}))+\frac{1}{2}(x^{k}-x^{k+1})\right\|^2+\frac{L_1}{4}\|x^{k}-x^{k+1}\|^2&\\
	-L_2\left\|\frac{1}{L_2}(\nabla h_2(z^*)-\nabla h_2(z^{k}))+\frac{1}{2}(z^{k}-z^{k+1})\right\|^2+\frac{L_2}{4}\|z^{k}-z^{k+1}\|^2&.
	\end{align*}
Finally, by using the monotonicity of $(M_2^k)_{k\geq 0}$ and of $(c_k)_{k\geq 0}$, we obtain
	\begin{equation}\label{eq:Prox-AMA-inequality-monotonicity-M2k}
	\|p^{k+1}-p^*\|^2+c_{k+1}\|z^{k+1}-z^*\|^2_{M_2^{k+1}} \leq 
	\|p^k-p^*\|^2+c_k\|z^{k}-z^*\|^2_{M_2^k} - R_k,
	\end{equation} 
	where 
	\begin{align*}
	R_k:= & \ c_k\left(2\gamma-c_k\|A\|^2\right)\|x^{k+1}-x^*\|^2 +c_k^2\|Bz^{k+1}-Bz^*\|^2 + \\
	& \ c_k\|z^{k}-z^{k+1}\|^2_{M_2^k-\frac{L_2}{2}\Id}+c_k\|x^{k}-x^{k+1}\|^2_{M_1^k-\frac{L_1}{2}\Id} +\\
	& \ 2c_kL_1\left\|\frac{1}{L_1}(\nabla h_1(x^*)-\nabla h_1(x^{k}))+\frac{1}{2}(x^{k}-x^{k+1})\right\|^2 +\\
	& \ 2c_kL_2\left\|\frac{1}{L_2}(\nabla h_2(z^*)-\nabla h_2(z^{k}))+\frac{1}{2}(z^{k}-z^{k+1})\right\|^2.
	\end{align*}
If $L_1 =0$ (and, consequently, $\nabla h_1$ is constant) and $L_2 >0$, then, by using the same arguments, we obtain again \eqref{eq:Prox-AMA-inequality-monotonicity-M2k}, but with
\begin{align*}
	R_k:= & \ c_k\left(2\gamma-c_k\|A\|^2\right)\|x^{k+1}-x^*\|^2 +c_k^2\|Bz^{k+1}-Bz^*\|^2 + \\
	& \ c_k\|z^{k}-z^{k+1}\|^2_{M_2^k-\frac{L_2}{2}\Id}+c_k\|x^{k}-x^{k+1}\|^2_{M_1^k} +\\
	& \ 2c_kL_2\left\|\frac{1}{L_2}(\nabla h_2(z^*)-\nabla h_2(z^{k}))+\frac{1}{2}(z^{k}-z^{k+1})\right\|^2.
	\end{align*}

If $L_2 = 0$ (and, consequently, $\nabla h_2$ is constant) and $L_2 >0$, then, by using the same arguments, we obtain again \eqref{eq:Prox-AMA-inequality-monotonicity-M2k}, but with
\begin{align*}
	R_k:= & \ c_k\left(2\gamma-c_k\|A\|^2\right)\|x^{k+1}-x^*\|^2 +c_k^2\|Bz^{k+1}-Bz^*\|^2 + \\
	& \ c_k\|z^{k}-z^{k+1}\|^2_{M_2^k}+c_k\|x^{k}-x^{k+1}\|^2_{M_1^k-\frac{L_1}{2}\Id} +\\
	& \ 2c_kL_1\left\|\frac{1}{L_1}(\nabla h_1(x^*)-\nabla h_1(x^{k}))+\frac{1}{2}(x^{k}-x^{k+1})\right\|^2.
	\end{align*}
Relation \eqref{eq:Prox-AMA-inequality-monotonicity-M2k} follows even if $L_1=L_2=0$, but with
\begin{align*}
	R_k:= & \ c_k\left(2\gamma-c_k\|A\|^2\right)\|x^{k+1}-x^*\|^2 +c_k^2\|Bz^{k+1}-Bz^*\|^2 + \\
	& \ c_k\|z^{k}-z^{k+1}\|^2_{M_2^k}+c_k\|x^{k}-x^{k+1}\|^2_{M_1^k}.
	\end{align*}
Notice that, due to $M_1^{k}-\frac{L_1}{2}\Id \in \mathcal{S}_+(\mathcal{H})$ and $M_2^{k}-\frac{L_2}{2}\Id \in \mathcal{S}_+(\mathcal{G})$, all summands in $R_k$ are nonnegative.

 Let be $N \geq 0$ fixed. By telescoping we obtain
	\begin{equation*}
	\|p^{N+1}-p^*\|^2+c_N\|z^{N+1}-z^*\|^2_{M_2^{N+1}} \leq \|p^0-p^*\|^2+c_0\|z^0-z^*\|_{M_2^0}-\sum_{k=0}^{N}	R_k.
	\end{equation*}
From \eqref{eq:Prox-AMA-inequality-monotonicity-M2k} we also obtain that
\begin{equation}\label{eq:Prox-AMA_lim_p^k_z^k}
	\exists\lim_{k \rightarrow \infty} \left (\|p^k-p^*\|^2+c_k\|z^{k}-z^*\|^2_{M_2^{k}}\right),
	\end{equation}
thus  $(p^k)_{k \geq 0}$ is bounded, and $\sum_{k \geq 0} R_k < + \infty$.

Taking \eqref{as:Prox-AMA:c_k} into account we have $c_k(2\gamma-c_k\|A\|^2)\geq \varepsilon^2\|A\|^2$ for all $k \geq 0$. Therefore
	\begin{equation}
	\sum_{k\geq 0}\|x^{k+1}-x^*\|^2< +\infty, \quad \sum_{k\geq 0}\|Bz^{k+1}-Bz^*\|^2< +\infty
	\end{equation}
	and
	\begin{equation}\label{eq:Prox-AMA-sums-x^k-z^k}
	\sum_{k\geq 0}\|z^{k+1}-z^{k}\|^2_{M_2^{k}-\frac{L_2}{2}\Id}< +\infty.
	\end{equation}
From here we have
	\begin{align}
	x^k \rightarrow x^*, \quad Bz^k \rightarrow Bz^* \ (k \rightarrow +\infty),\label{eq:Prox-AMA-conv-Bz^k}
	\end{align}
which, by using \eqref{eq:prox-AMA-h-p-Update} and \eqref{eq:Prox-AMA-const},  lead to
	\begin{equation} \label{eq:Prox-AMA-conv-p^k-p^k-1}
	p^{k}-p^{k+1} \rightarrow 0 \ (k \rightarrow +\infty). 
	\end{equation}
	
Suppose that assumption (i) holds true, namely, that there exists $\alpha > 0$ such that $M_2^k-\frac{L_2}{2}\Id \in \mathcal{P}_\alpha(\GG)$ for all $k \geq 0$. From \eqref{eq:Prox-AMA_lim_p^k_z^k} it follows that $(z^k)_{k \geq 0}$ is bounded, while \eqref{eq:Prox-AMA-sums-x^k-z^k} ensures that 
	\begin{equation}\label{eq:Prox-AMA-conv-z^k-z^k-1}
	z^{k+1}-z^{k} \to 0 \ (k \rightarrow +\infty).
	\end{equation}

In the following we show that each weak sequential cluster point of $(x^k,z^k,p^k)_{k\geq 0}$ (notice that the sequence is bounded) is a saddle point of $L$. Let be $(\bar{z},\bar{p}) \in  \GG \times \mathcal{K}$ such that the subsequence $(x^{k_j}, z^{k_j}, p^{k_j})_{j\geq 0}$ converges weakly  to $(x^*,\bar{z},\bar{p})$ as $j\rightarrow+\infty$.
	From \eqref{eq:Prox-AMA-ocAlg1} we have
	\begin{equation*}
	A^*p^{k_j}-\nabla h_1(x^{k_j}) + M_1^{k_j}(x^{k_j}-x^{k_j+1})  \in \partial f(x^{k_j+1}) \ \forall j \geq 1.
	\end{equation*}
Since $x^{k_j}$ converges strongly to $x^*$ and $p^{k_j}$ converges weakly to a $\bar{p}$ as $j \to +\infty$, using the continuity of $\nabla h_1$ and the closedness of the graph of the convex subdifferential of $f$ in the strong-weak topology (see \cite[Proposition 20.33]{baco17}), it follows
	\begin{equation*}
	A^*\bar{p}-\nabla h_1(x^*)  \in \partial f(x^*).
	\end{equation*}
	From \eqref{eq:Prox-AMA-ocAlg2} we have for all $j \geq 0$
	\begin{equation*}
	B^*p^{k_j}-\nabla h_2(z^{k_j}) + c_{k_j}B^*(-Ax^{k_j+1}-Bz^{k_j+1}+b)+M_2^{k_j}(z^{k_j}-z^{k_j+1}) \in \partial g(z^{k_j+1}),
	\end{equation*}
which is equivalent to
	\begin{align*}
	B^*p^{k_j}+\nabla h_2(z^{k_j+1})-\nabla h_2(z^{k_j}) +  c_{k_j}B^*(-Ax^{k_j+1}-Bz^{k_j+1}+b)+M_2^{k_j}(z^{k_j}-z^{k_j+1})\nonumber\\ 
	\in \partial (g+h_2)(z^{k_j+1})
	\end{align*}
and further to
	\begin{align}\label{eq34}
	z^{k_j+1} \in \partial(g+h_2)^*\Big (B^*p^{k_j} & +  \nabla h_2(z^{k_j+1})-\nabla h_2(z^{k_j}) \nonumber \\
& + c_{k_j}B^*(-Ax^{k_j+1}-Bz^{k_j+1}+b)+M_2^{k_j}(z^{k_j}-z^{k_j+1})\Big).
	\end{align}
By denoting for all $j \geq 0$
	\begin{align*}
	v^j &:=z^{k_j+1}, u^j : =p^{k_j},\\
w^j& :=\nabla h_2(z^{k_j+1})-\nabla h_2(z^{k_j})+c_{k_j}B^*(-Ax^{k_j+1}-Bz^{k_j+1}+b)+M_2^{k_j}(z^{k_j}-z^{k_j+1}),
	\end{align*}
\eqref{eq34} reads
	\begin{equation*}
	v^j \in \partial (g+h_2)^*(B^*u^j+w^j) \ \forall j \geq 0.
	\end{equation*}
According to \eqref{eq:Prox-AMA-conv-z^k-z^k-1} we have
	\begin{align*}
	v^j &\rightharpoonup \bar{z}, \quad u^j \rightharpoonup \bar{p} \ (j \rightarrow +\infty)
	\end{align*}
and, 	by taking into account \eqref{eq:Prox-AMA-conv-Bz^k}, it holds
\begin{equation*}
	Bv^j \to B \bar z =  Bz^* \ (j \rightarrow +\infty).
	\end{equation*}
Combining \eqref{eq:Prox-AMA-conv-z^k} with the Lipschitz continuity of $\nabla h_2$,  \eqref{eq:Prox-AMA-conv-p^k-p^k-1}, \eqref{eq:Prox-AMA-conv-z^k-z^k-1}
	and \eqref{eq:prox-AMA-h-p-Update}, one can easily see that
	\begin{equation*}
	w^j \to 0 \ (j \rightarrow +\infty).
	\end{equation*}
Due to the monotonicity of the subdifferential  it holds
	\begin{equation*}
	\langle v^j-v,B^*u^j+w^j-u\rangle \geq 0 \ \forall (u,v) \mbox{ in the graph of }\partial(g+h_2)^* \mbox{ and }\ \forall j\geq 0,
	\end{equation*}
	which is equivalent to
	\begin{equation*}
		\langle Bv^j-Bv,u^j \rangle+ \langle v^j-v,w^j-u\rangle \geq 0 \ \forall (u,v) \mbox{ in the graph of }\partial(g+h_2)^* \mbox{ and }\ \forall j\geq 0.
	\end{equation*}
We let $j$ converge to $+\infty$ and receive
	\begin{equation*}
	\langle \bar{z}-v,B^*\bar{p} \rangle- \langle \bar{z}-v,u\rangle \geq 0 \ \ \forall (u,v) \mbox{ in the graph of }\partial(g+h_2)^*
	\end{equation*}
or, equivalently,
	\begin{equation*}
	\langle \bar{z}-v,B^*\bar{p} -u \rangle \geq 0 \ \ \forall (u,v) \mbox{ in the graph of }\partial(g+h_2)^*.
	\end{equation*}
The maximal monotonicity of the convex subdifferential of $(g+h_2)^*$ ensures that $\bar{z} \in \partial (g+h_2)^*(B^*\bar{p})$, which is the same as 
	$B^*\bar{p} \in \partial (g+h_2)(\bar{z})$. In other words, $B^*\bar{p}-\nabla h_2(\bar{z}) \in \partial g(\bar{z})$. Finally, from \eqref{eq:prox-AMA-h-p-Update} and \eqref{eq:Prox-AMA-conv-p^k-p^k-1} it follows that $Ax^* + B\bar z = b$. In conclusion, $(x^*, \overline z,\bar{p})$ is a saddle point of the Lagrangian $L$. 
	
In the following we show that sequence $(x^k,z^k,p^k)_{k \geq 0}$ converges weakly. Let $(x^*,z_1,p_1)$ and $(x^*,z_2,p_2)$ be two weak sequential cluster points $(x^k,z^k,p^k)_{k \geq 0}$. 
Then there exists $(k_s)_{s \geq 0}$, $k_s \to + \infty$ as $s \to + \infty$, such that the subsequence $(x^{k_s}, z^{k_s}, p^{k_s})_{s \geq 0}$ converges weakly to $(x^*,z_1,p_1)$ as $s \to + \infty$. 
Furthermore there exists $(k_t)_{t \geq 0}$, $k_t \to + \infty$ as $t \to + \infty$, such that that a subsequence $(x^{k_t}, z^{k_t}, p^{k_t})_{t \geq 0}$ converges weakly to $(x^*,z_2,p_2)$ as $t \to + \infty$. 
As seen before, $(x^*,z_1,p_1)$ and $(x^*,z_2,p_2)$ are both saddle points of the Lagrangian $L$.

From \eqref{eq:Prox-AMA_lim_p^k_z^k}, which is fulfilled for every saddle point of the Lagrangian $L$, we obtain
		\begin{equation}\label{eq:Prox-AMA_lim_p^k_z^k2}
		\exists \lim_{k \to + \infty}(\|p^k-p_1\|^2-\|p^k-p_2\|^2+c_k\|z^{k}-z_1\|^2_{M_2^k}-c_k\|z^{k}-z_2\|^2_{M_2^k}):=T.
		\end{equation}
For all $k \geq 0$ we have
		\begin{align*}
		\|p^k-p_1\|^2-\|p^k-p_2\|^2+c_k\|z^{k}-z_1\|^2_{M_2^k}-c_k\|z^{k}-z_2\|^2_{M_2^k} & = \\
		\|p_2-p_1\|^2+2\langle p_k-p_2, p_2-p_1 \rangle + c_k\|z_2-z_1\|_{M_2^k}^2+2c_k\langle z_{k}-z_2, z_2-z_1 \rangle_{M_2^k}. &
		\end{align*}
 Since $M_2^k \geq \left (\alpha + \frac{L_2}{2} \right) \Id $ for all $k \geq 0$ and $(M_2^k)_{k \geq 0}$
is a monotone sequence of symmetric operators, there exists a symmetric operator $M \geq \left (\alpha + \frac{L_2}{2} \right) \Id$ such that  $(M_2^k)_{k \geq 0}$ converges pointwise to $M$ in the strong topology as $k \rightarrow +\infty$. Furthermore, let $c:=\lim_{k \rightarrow +\infty} c_k >0$.
Taking the limits in \eqref{eq:Prox-AMA_lim_p^k_z^k2} along the subsequences $(k_s)_{s \geq 0}$ and $(k_t)_{t \geq 0}$, it yields
		\begin{equation*}
		T=-\|p_2-p_1\|^2-c\|z_2-z_1\|^2_M
		\end{equation*}
		and
		\begin{equation*}
		T=\|p_2-p_1\|^2+c\|z_2-z_1\|^2_M,
		\end{equation*}
thus
		\begin{equation*}
		\|p_2-p_1\|^2+c\|z_2-z_1\|^2_M=0.
		\end{equation*} 
It follows that $p_1=p_2$ and $z_1=z_2$, thus $(x^k, z^k, p^k)_{k\geq 0}$ converges weakly to a saddle point of the Lagrangian $L$.
	
Assume now that condition (ii) holds, namely, that there exists $\beta > 0$ such that $B^*B \in \mathcal{P}_\beta(\HH)$. Then $\beta\|z_1-z_2\|^2 \leq \|Bz_1-Bz_2\|^2$ for all $z_1, z_2 \in \GG$, which means that, if $(x_1^*,z_1^*,p_1^*)$ and $(x_2^*,z_2^*,p_2^*)$ are two saddle points of the Lagrangian $L$, then $x_1^*=x_2^*$ and $z_1^*=z_2^*$. 

For the saddle point $(x^*, z^*, p^*)$ of the Lagrangian $L$ we fixed at the beginning of the proof and the generated sequence $(x^k, z^k, p^k)_{k \geq 0}$ we receive because of \eqref{eq:Prox-AMA-conv-Bz^k} that
	\begin{equation}\label{eq:Prox-AMA-conv-z^k}
		x^k \rightarrow x^*, \quad z^k \rightarrow z^*, \quad p^{k} - p^{k+1} \rightarrow 0 \ (k \rightarrow +\infty).
	\end{equation}
Moreover,
\begin{align*}
\exists\lim_{k \rightarrow \infty} \|p^k-p^*\|^2.
\end{align*}
The remainder of the proof follows in analogy to the one given under assumption (i).
\end{proof}

If $h_1=0$ and $h_2=0$,  and $M_1^k=0$ and $M_2^k = 0$ for all $k \geq 0$,  then the Proximal AMA method becomes the AMA method as it has been proposed by Tseng in \cite{tseng91}. According to Theorem \ref{th:Prox-AMA-convergence} (for $L_1=L_2=0$), the generated sequence converges weakly to a saddle point of the Lagrangian, if there exists $\beta >0$ such that $B^*B\in \mathcal{P}_{\beta}(\mathcal{G})$. In finite dimensional spaces this condition reduces to assuming that $B$ is injective.

\section{Numerical experiments}\label{sec3}

In this section we compare the numerical performances of AMA  and  Proximal AMA on two applications in  image processing and machine learning. The numerical experiments were performed on a computer with an Intel Core i5-3470 CPU and 8 GB DDR3 RAM. 

\subsection{Image denoising and deblurring}\label{subsec31}

We addressed an image denoising and deblurring problem formulated as a nonsmooth convex optimization problem (see \cite{bohe14, hend14, ruos92})
\begin{equation}\label{eq:Problem_Image}
\inf_{x \in \R^n} \left\{\frac{1}{2}\|Ax-b\|^2 + \lambda \text{TV}(x)\right\},
\end{equation}
where $A \in \R^{n\times n}$ represents a blur operator, $b \in \R^n$ is a given blurred and noisy image, $\lambda>0$ is a regularization parameter and $\text{TV}:\R^n\to \R$ is a discrete total variation functional. 
The vector $x \in \R^n$ is the vectorized image $X \in \R^{M \times N}$, where $n=MN$ and $x_{i,j} := X_{i,j}$ stands for the normalized value of the pixel in the $i$-th row and the $j$-th column, for $1 \leq i \leq M, 1 \leq j \leq N$. 

Two choices have been considered for the discrete total variation, namely, the  isotropic total variation $\text{TV}_{\text{iso}}:\R^n\to \R,$
\[\text{TV}_{\text{iso}}(x)=\sum_{i=1}^{M-1}\sum_{j=1}^{N-1}\sqrt{(x_{i+1,j}-x_{i,j})^2+(x_{i,j+1}-x_{i,j})^2}
+\sum_{i=1}^{M-1}|x_{i+1,N}-x_{i,j}|+\sum_{j=1}^{N-1}|x_{M,j+1}-x_{M,j}|,
\]
and the anisotropic total variation $\text{TV}_{\text{aniso}}:\R^n\to \R,$
\[\text{TV}_{\text{aniso}}(x)=\sum_{i=1}^{M-1}\sum_{j=1}^{N-1}|x_{i+1,j}-x_{i,j}|+|x_{i,j+1}-x_{i,j}|
+\sum_{i=1}^{M-1}|x_{i+1,N}-x_{i,j}|+\sum_{j=1}^{N-1}|x_{M,j+1}-x_{M,j}|.
\]
Consider the linear operator $L:\R^n\to \R^n \times \R^n, x_{i,j} \mapsto \left(L_1x_{i,j},L_2x_{i,j}\right)$, where
\[L_1x_{i,j}=\begin{cases}
x_{i+1,j}-x_{i,j}, & \text{if } i<M\\
0, &\text{if } i=M
\end{cases}
\text{ and }
L_2x_{i,j}=\begin{cases}
x_{i,j+1}-x_{i,j}, & \text{if } j<N\\
0, &\text{if } j=N
\end{cases}
\]
One can easily see that $\|L\|^2\leq8$. The optimization problem \eqref{eq:Problem_Image} can be written as
\begin{equation}\label{eq:Problem_Image_Primal}
\inf_{x \in \R^n}\left\{f(Ax)+g(Lx)\right\},
\end{equation}
where $f:\R^n \to \R, f(x)=\frac{1}{2}\|x-b\|^2$, and $g:\R^n \times \R^n \to \R$, 
$g(y,z)=\lambda\|(y,z)\|_1$, for the anisotropic total variation, and $g(y,z)=\lambda\|(y,z)\|_{\times}:=\lambda\sum_{i=1}^M\sum_{j=1}^N\sqrt{y_{i,j}^2+z_{i,j}^2}$, for the isotropic total variation.

We will solve the Fenchel dual problem of \eqref{eq:Problem_Image_Primal} by AMA and Proximal AMA and will determine in this way an optimal solution of the primal problem, too. 
The reason for this strategy is that the Fenchel dual problem of \eqref{eq:Problem_Image_Primal}  is a convex optimization problem with two-block separable linear constraints and objective function.

Indeed, the Fenchel dual problem of \eqref{eq:Problem_Image_Primal} is (see \cite{baco17,b-hab})
\begin{align}\label{eq:Problem_Image_Dual}
\inf_{p \in \R^n, q \in \R^n \times \R^n}\left\{f^*(p) + g^*(q)\right\}.\\
\text{s.t. }A^* p+L^* q=0. \nonumber
\end{align}
Since $f$ and $g$ have full domains, strong duality for \eqref{eq:Problem_Image_Primal}-\eqref{eq:Problem_Image_Dual} holds. 

We notice that $f^*(p)=\frac{1}{2}\|p\|^2+\langle p, b\rangle$ for all $p \in \R^n$, hence $f^*$ is $1$-strongly convex. We choose $M_1^k = 0$ and $M_2^k=\frac{1}{\sigma_k}\text{I}-c_kL^* L$ (see Remark \ref{prox-step}) for every $k \geq 0$. The iterative scheme of Proximal AMA becomes for all $k \geq 0$:
	\begin{align*}
	p^{k+1}&=Ax^k-b\\
	q^{k+1}&=\text{Prox}_{\sigma_kg^*}\left(q^{k}+\sigma_kc_kL(-A^* p^{k+1} - L^*q^k)+\sigma_kL(x^k)\right)\\
	x^{k+1}&=x^k+c_k(-A^* p^{k+1}-L^* q^{k+1}).
	\end{align*}

In the case of the anisotropic total variation, the conjugate of $g$ is the indicator function of the set $[-\lambda,\lambda]^n\times[-\lambda,\lambda]^n$, thus $\text{Prox}_{\sigma_kg^*}$ is the projection operator $\mathcal{P}_{[-\lambda,\lambda]^n\times[-\lambda,\lambda]^n}$ on the set $[-\lambda,\lambda]^n\times[-\lambda,\lambda]^n$. The iterative scheme becomes for all $k \geq 0$:
\begin{align*}
	p^{k+1}&=Ax^k-b\\
	(q_1^{k+1},q_2^{k+1})&=\mathcal{P}_{[-\lambda,\lambda]^n\times[-\lambda,\lambda]^n}\left((q_1^{k},q_2^{k})+c_k\sigma_k(-LA^* p^{k+1} - LL^* (q_1^{k},q_2^{k}))+\sigma_kLx^k\right)\\
	x^{k+1}&=x^k+c_k\left(-A^* p^{k+1}-L^*(q_1^{k+1},q_2^{k+1})\right).
\end{align*}

In  the case of the isotropic total variation, the conjugate of $g$ is the indicator function of the set
$S:=\left\{(v,w)\in \R^n\times \R^n:\max_{1\leq i\leq n}\sqrt{v_i^2+w_i^2}\leq \lambda \right\}$, thus 
$\text{Prox}_{\sigma_kg^*}$ is the projection operator $P_S : \R^n \times \R^n \rightarrow S$ on $S$, which reads
\[
(v_i,w_i) \mapsto \lambda\frac{(v_i,w_i)}{\max\left\{\lambda, \sqrt{v_i^2+w_i^2}\right\}}, \quad i=1,...,n.
\]
The iterative scheme becomes for all $k \geq 0$:
\begin{align*}
	p^{k+1}&=Ax^k-b\\
	(q_1^{k+1},q_2^{k+1})&=P_S\left((q_1^{k},q_2^{k})+c_k\sigma_k(-LA^* p^{k+1} - LL^* (q_1^{k},q_2^{k}))+\sigma_kLx^k\right)\\
	x^{k+1}&=x^k+c_k\left(-A^* p^{k+1}-L^*(q_1^{k+1},q_2^{k+1})\right).
\end{align*}

\begin{figure}[ht]
	\centering
	\captionsetup[subfigure]{position=top}
	\subfloat[Original image "office\_4"]{\includegraphics*[width=0.32\linewidth]{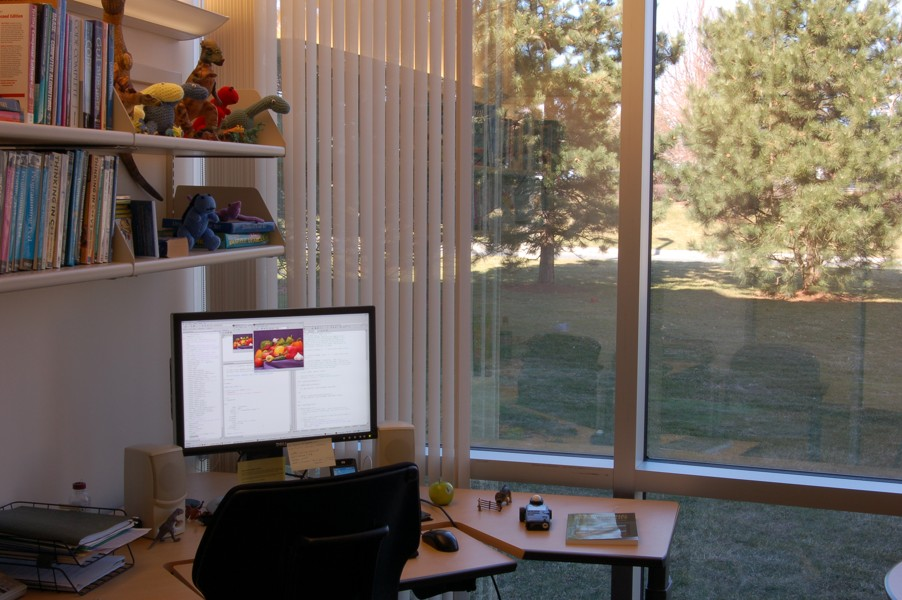}} \hspace{0.2cm}
	\subfloat[Blurred and noisy image]{\includegraphics*[width=0.32\linewidth]{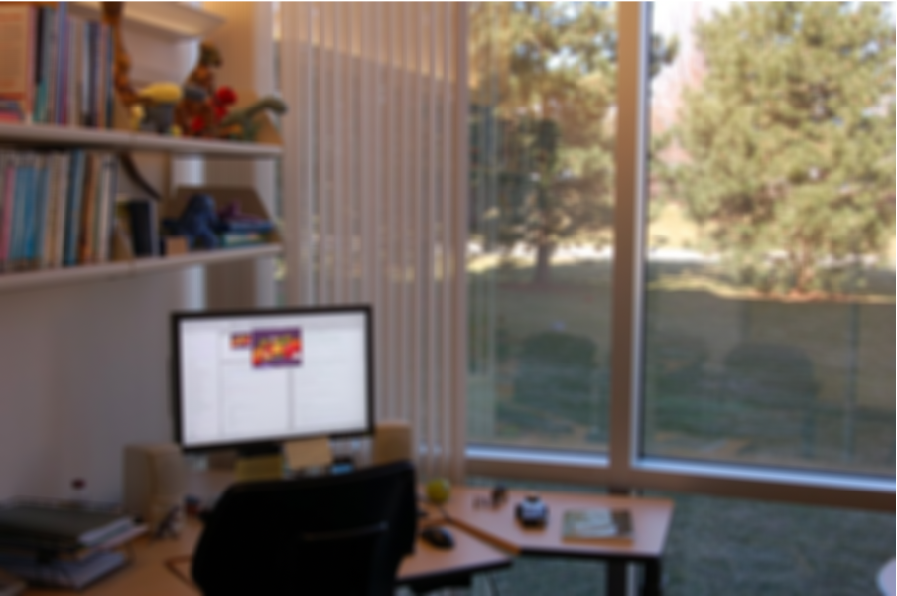}} \hspace{0.2cm}
	\subfloat[Reconstructed image]{\includegraphics*[width=0.32\linewidth]{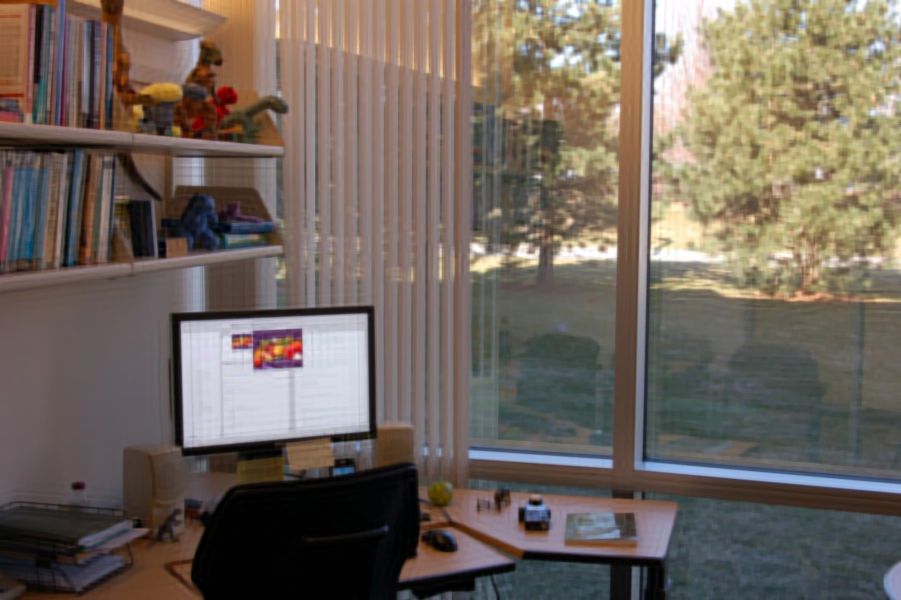}}
	\caption{\small The original image, the blurred and noisy image and the reconstructed image after 50 seconds cpu time.}
	\label{fig1}	
\end{figure}

We compared the Proximal AMA method with Tseng's AMA method. While in Proximal AMA a closed formula is available for the computation of $(q_1^{k+1},q_2^{k+1})_{k \geq 0}$, in AMA we solved the resulting optimization subproblem 
\begin{equation*}
(q_1^{k+1},q_2^{k+1})=\argmin_{q_1,q_2}\left\{g^*(q_1,q_2)-\langle x^{k+1},L^*(q_1,q_2)\rangle+\frac{1}{2}c_k\|A^*p^{k+1}+L^*(q_1,q_2)\|^2\right\}
\end{equation*}
in every iteration $k \geq 0$ by making some steps of the FISTA method (\cite{bete09}).

\begin{figure}[ht]
	\centering
	\captionsetup[subfigure]{position=top}
	{\includegraphics*[width=0.48\linewidth]{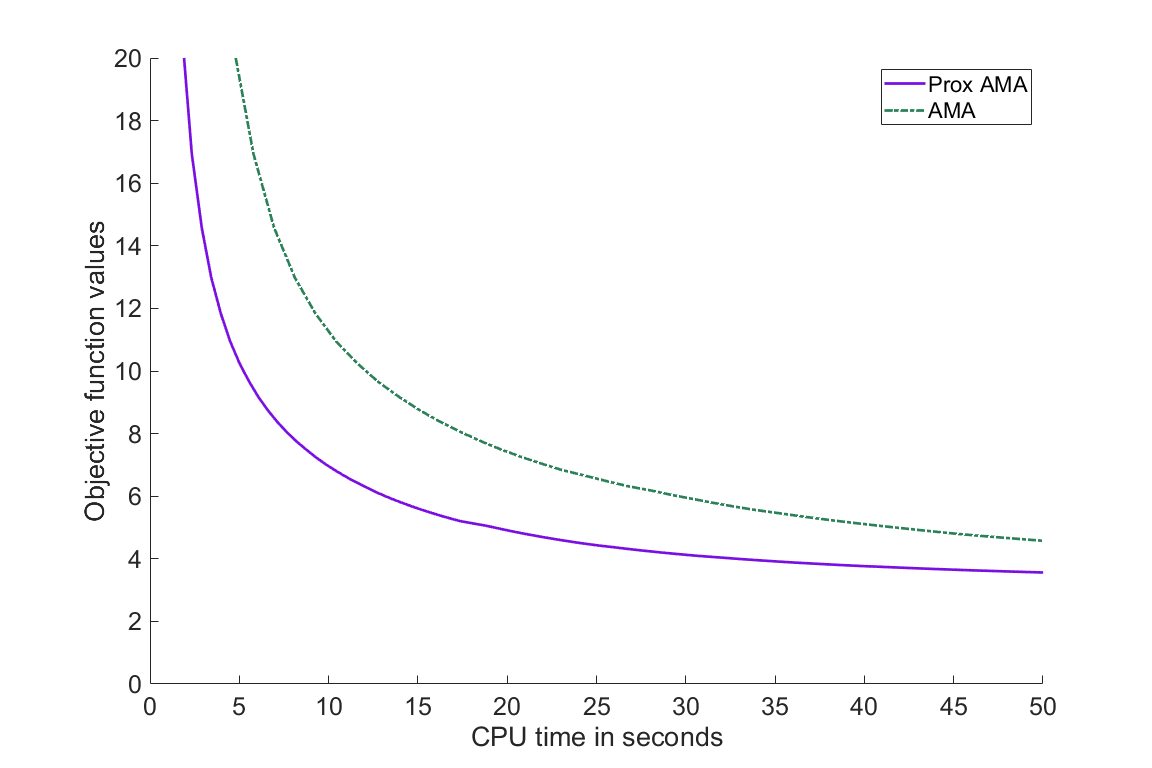}} \hspace{0.2cm}
	{\includegraphics*[width=0.48\linewidth]{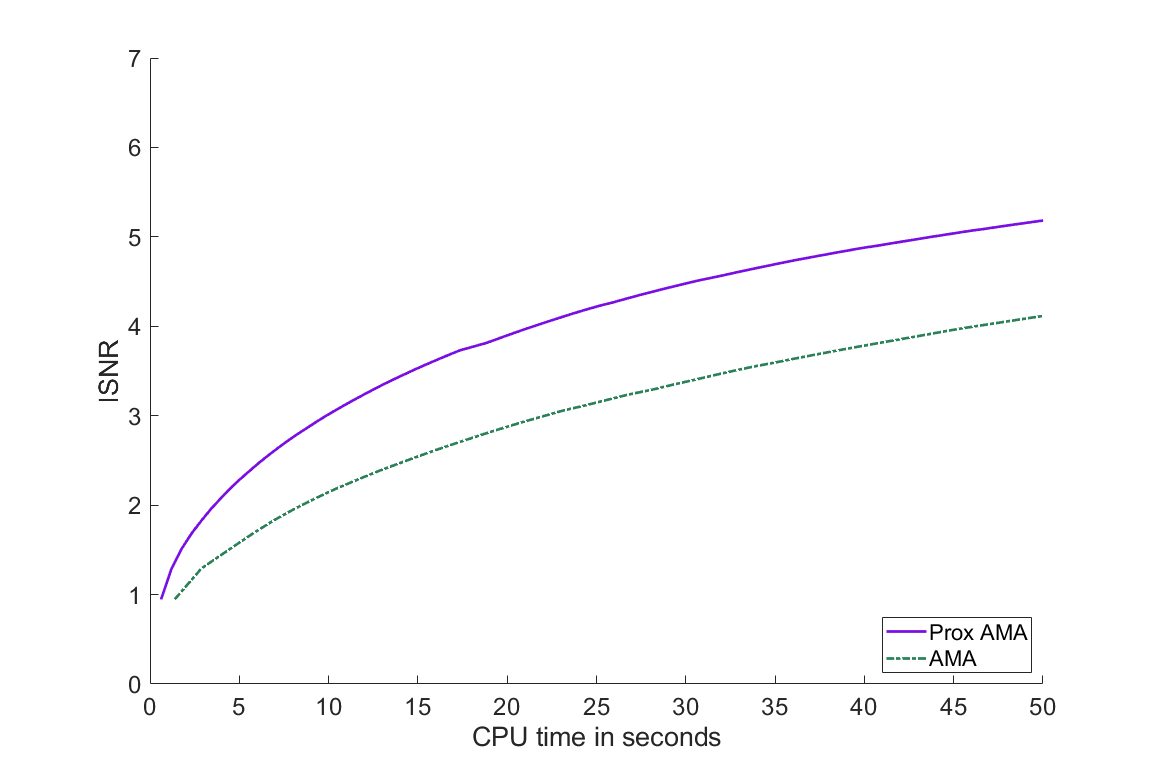}} 
	\caption{\small The objective function values and the ISNR values for the anisotropic TV and $\lambda=5\cdot10^{-5}$.}	
	\label{fig2}
\end{figure}

\begin{figure}[ht]
	\centering
	\captionsetup[subfigure]{position=top}
	{\includegraphics*[width=0.48\linewidth]{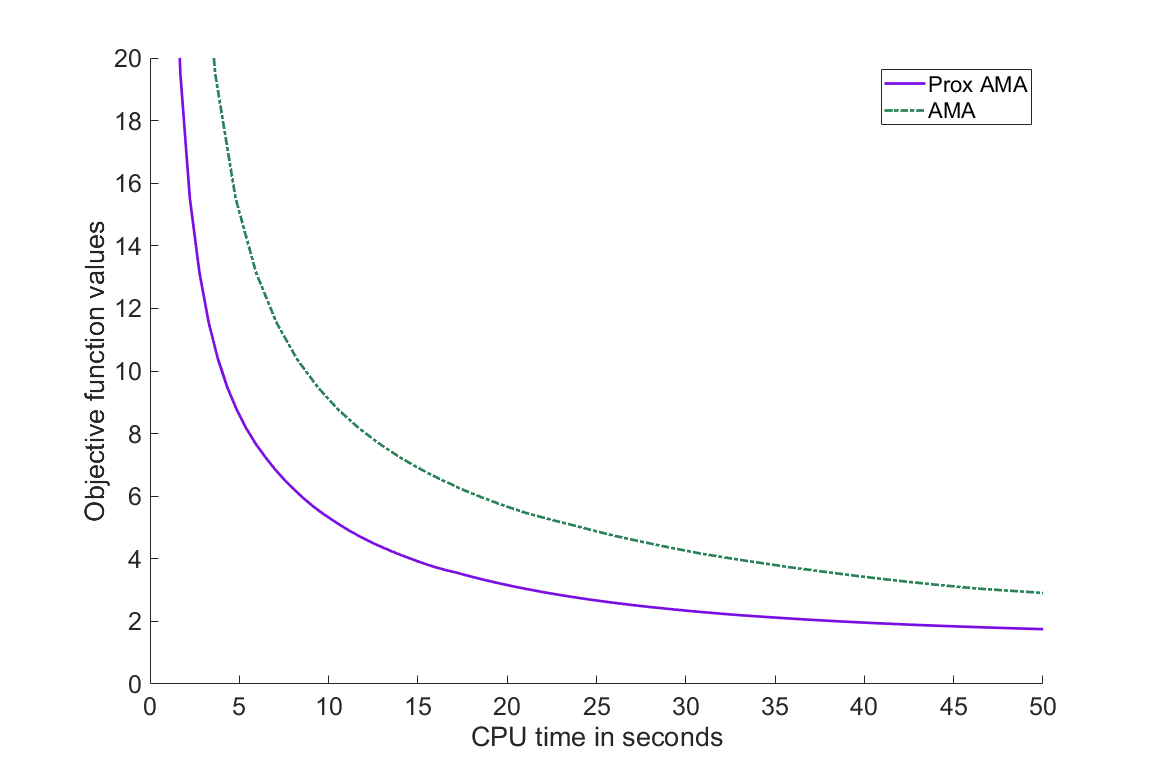}} \hspace{0.2cm}
	{\includegraphics*[width=0.48\linewidth]{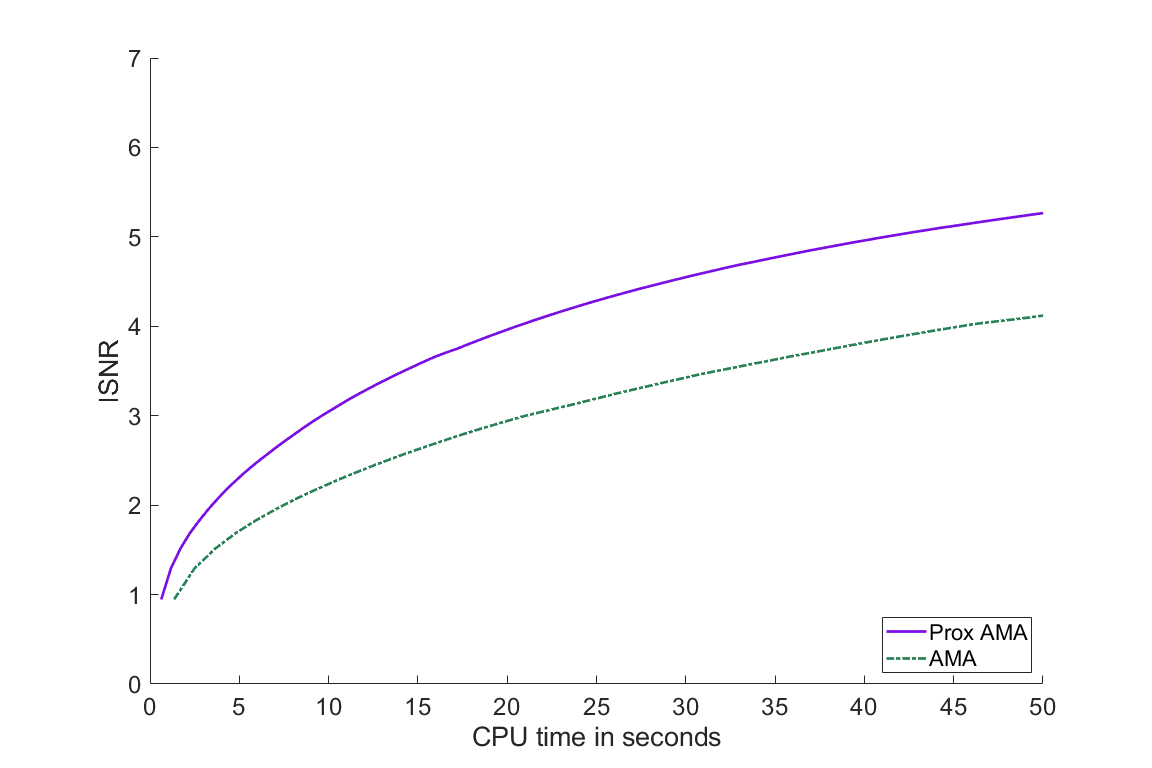}} 
	\caption{\small The objective function values and the ISNR values for the anisotropic TV and $\lambda=10^{-5}$.}	
	\label{fig3}
\end{figure}

\begin{figure}[ht]
	\centering
	\captionsetup[subfigure]{position=top}
	{\includegraphics*[width=0.48\linewidth]{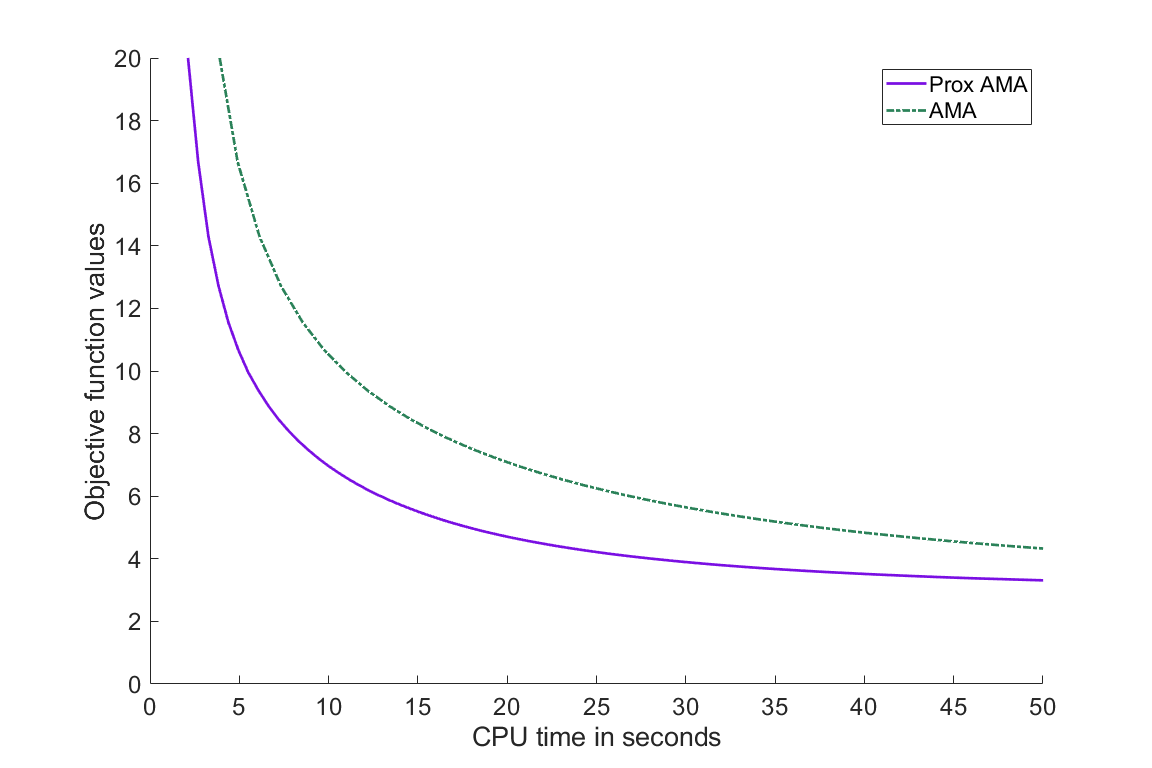}} \hspace{0.2cm}
	{\includegraphics*[width=0.48\linewidth]{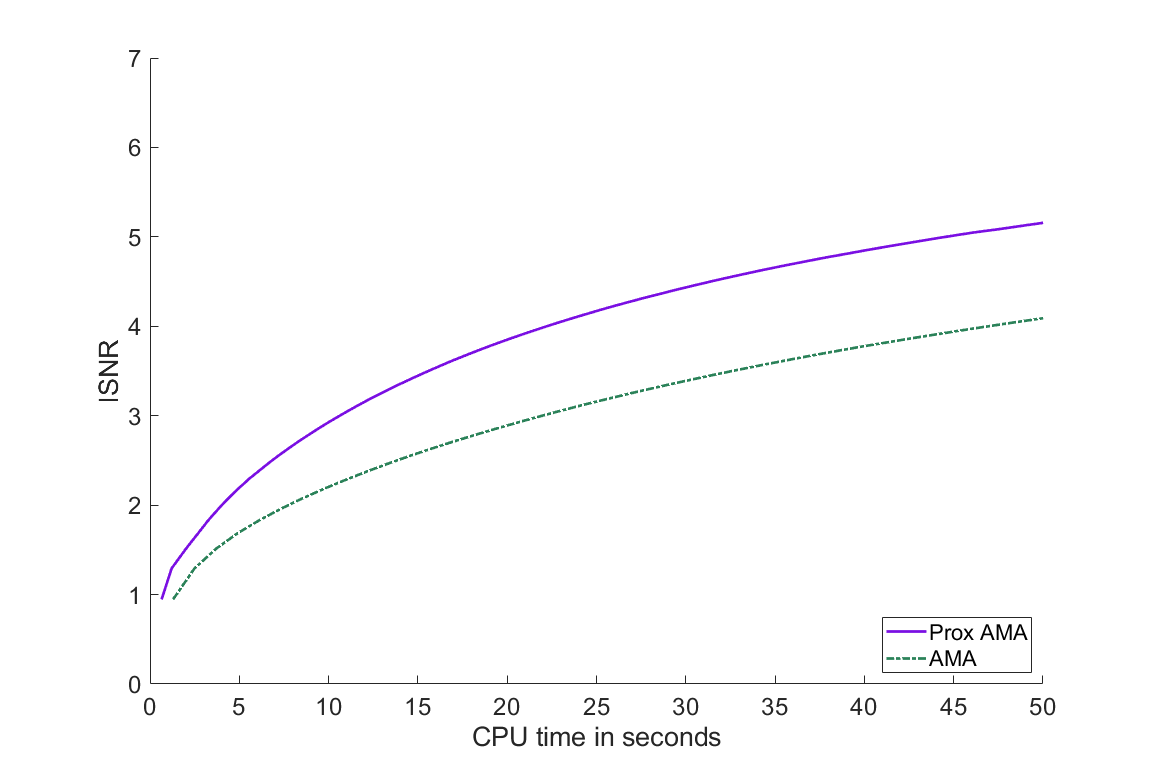}} 
	\caption{\small The objective function values and the ISNR values for the isotropic TV and $\lambda=5\cdot10^{-5}$.}	
	\label{fig4}
\end{figure}

\begin{figure}[ht]
	\centering
	\captionsetup[subfigure]{position=top}
	{\includegraphics*[width=0.48\linewidth]{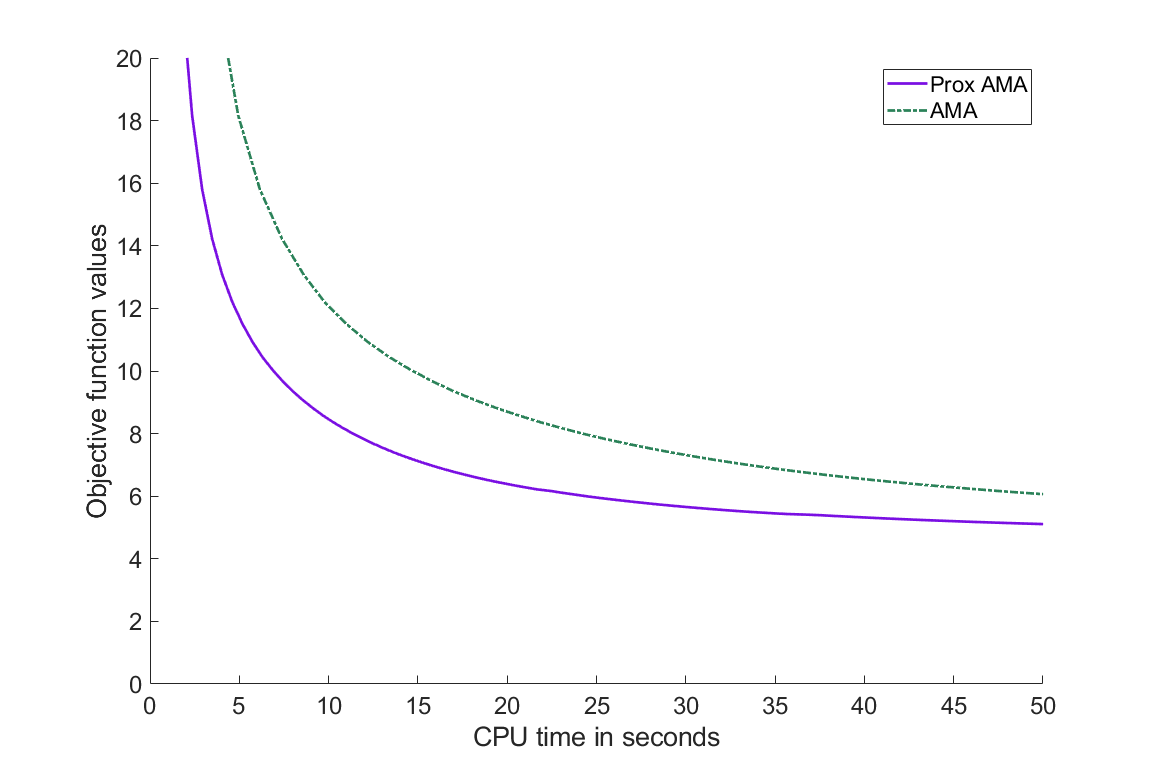}} \hspace{0.2cm}
	{\includegraphics*[width=0.48\linewidth]{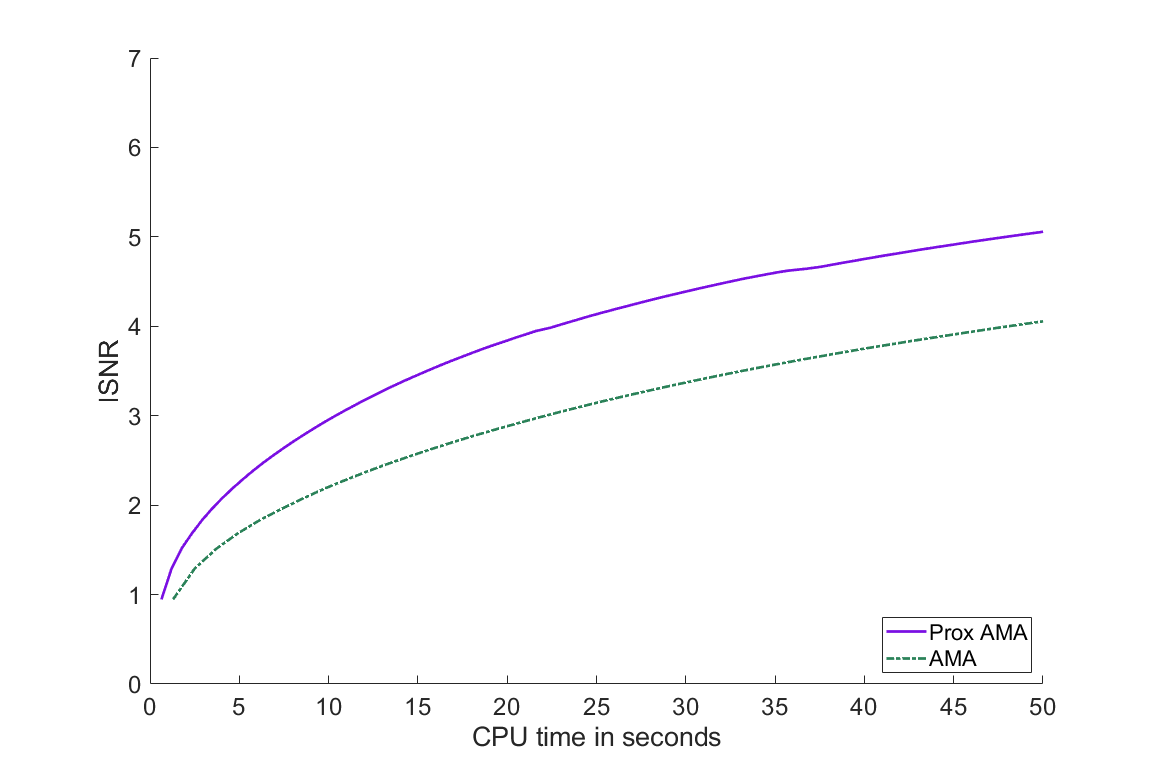}} 
	\caption{\small The objective function values and the ISNR values for the isotropic TV and $\lambda=10^{-4}$.}	
	\label{fig5}
\end{figure}

We used in our experiments a Gaussian blur of size $9 \times 9$ and standard deviation $4$, which led to an operator $A$ with $\|A\|^2=1$ and $A^*=A$. Furthermore, we added Gaussian white noise with standard deviation $10^{-3}$. 
We used for both algorithms a constant sequence of stepsizes $c_k=2 -10^{-7}$ for all $k \geq 0$. One can notice that  $(c_k)_{k \geq 0}$ fulfils \eqref{as:Prox-AMA:c_k}. In Proximal AMA we considered $\sigma_k=\frac{1}{8.00001 \cdot c_k}$ for all $k \geq 0$, 
which ensured that every matrix $M_2^k=\frac{1}{\sigma_k}\text{I}-c_kL^* L$ is positively definite for all $k \geq 0$. This is actually the case, if $\sigma_kc_k\|L\|^2<1$ for all $k \geq 0$.  In other words, we guaranteed that assumption (i) in Theorem \ref{th:Prox-AMA-convergence} is fulfilled. 

In the figures \ref{fig2} - \ref{fig5} we show how Proximal AMA and AMA perform when reconstructing the blurred and noisy MATLAB test image "office\_ 4" for different choices for the regularization parameter $\lambda$ and by considering both the anisotropic and isotropic 
total variation as regularization functionals. For all considered instances one can notice that Proximal AMA outperforms AMA in both the convergence behaviour of the sequence of the function values and of the sequence of ISNR (Improvement in Signal-to-Noise Ratio) values.

\subsection{Kernel based machine learning}\label{subsec32}

In this subsection we will describe the numerical experiments we carried out in the context of classifying images via support vector machines. 

The given data set consisting of $5570$ training images and $1850$ test images of size $28 \times 28$ was taken from the website \href{http://www.cs.nyu.edu/~roweis/data.html}{http://www.cs.nyu.edu/~roweis/data.html}. The problem we considered was to determine a decision function based on a pool of handwritten digits showing either the number five or the number six, labeled by $+1$ and $-1$, respectively (see Figure \ref{fig:classification-five-six}). To evaluate the quality of the decision function we compute the percentage of misclassified images of the test data set.
\begin{figure}[ht]	
	\centering
	\includegraphics*[width=0.35\textwidth]{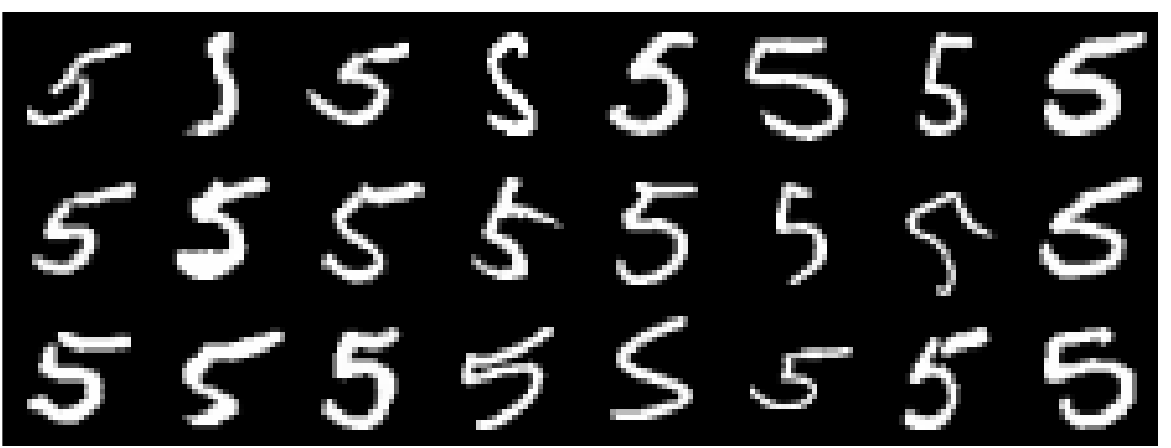} \hspace{2mm}
	\includegraphics*[width=0.35\textwidth]{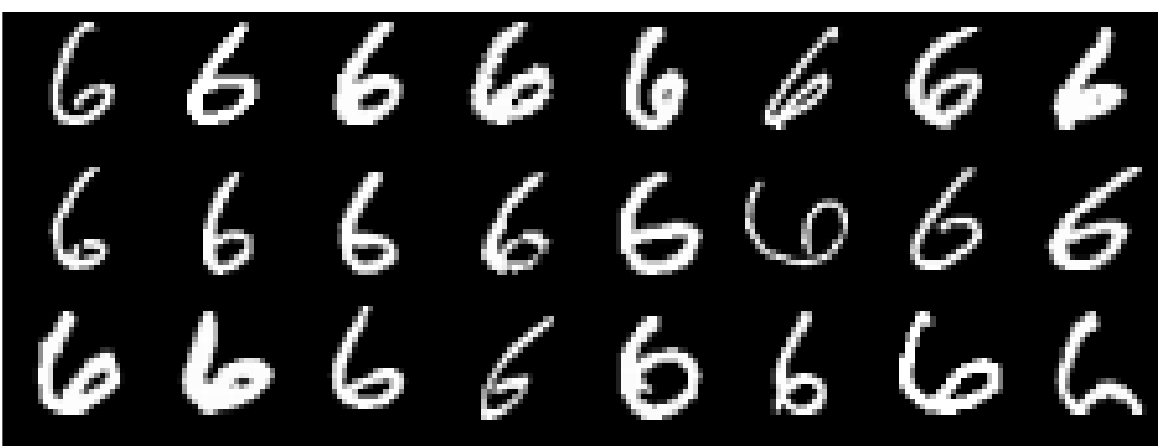}
	\caption{\small A sample of images belonging to the classes $+1$ and $-1$, respectively.}
	\label{fig:classification-five-six}	
\end{figure}

In order to describe the approach we used, let be
$$\Z=\{(X_1,Y_1),\ldots,(X_n,Y_n)\}\subseteq \R^d \times \{+1,-1\},$$
the given training data set. The decision functional $\verb"f"$ was assumed to be an element of the Reproducing Kernel Hilbert Space (RHKS) $\h_\kappa$, induced by the symmetric and finitely positive definite Gaussian kernel function
$$\kappa: \R^d \times \R^d \rightarrow \R, \  \kappa(x,y) = \exp \left( -\frac{\left\|x-y \right\|^2}{2 \sigma^2} \right).$$ 

By $K\in\R^{n\times n}$ we denoted the Gram matrix with respect to the training data set $\Z$,
namely, the symmetric and positive definite matrix with entries $K_{ij}=\kappa(X_i,X_j)$ for $i,j=1,\ldots,n$. To penalize the  deviation between the predicted value $\verb"f"(x)$ and the true value $y \in \{+1,-1\}$ we used the hinge loss functional $(x,y) \mapsto \max \{1-xy,0\}$. 

According to the Representer Theorem, the decision function $\verb"f"$ can be expressed as a kernel expansion in terms of the training data, i.e., $\verb"f"(\cdot) = \sum_{i=1}^n x_i \kappa(\cdot, X_i)$, where $x=(x_1,\ldots,x_n) \in \R^n$ is the optimal solution of the optimization problem
\begin{align}\label{eq50}
\min_{x \in \R^n}{\left\{\frac{1}{2}x^TKx + C\sum_{i=1}^n \max\{1-(Kx)_iY_i,0\} \right\}}.
\end{align}
Here, $C>0$ denotes the regularization parameter controlling the tradeoff between the loss function and the regularization term. Hence, in order to determine the decision function one has to solve the convex optimization problem \eqref{eq50}, which we write as
\begin{equation*}
\min_{x \in \R^n} \left\{f(x)+g(Kx)\right\}
\end{equation*}
or, equivalently,
\begin{align*}
\min_{x \in \R^n, z \in \R^n} &\left\{f(x)+g(z)\right\},\\
\text{s.t. } &Kx-z=0
\end{align*}
where $f:\R^n\to\R, f(x)=\frac{1}{2}x^TKx$, and $g:\R^n \to \R, g(z)=C\sum_{i=1}^{n}\max\{1-z_iY_i,0\}$.

Since the Gram matrix $K$ is positively definite, the function $f$ is $\lambda_{\min}(K)$-strongly convex, where $\lambda_{\min}(K)$ denotes the minimal eigenvalue of $K$, and differentiable, and it holds $\nabla f(x)= Kx$ for all $x \in \R^n$. For $p = (p_1,...,p_n) \in \R^n$, we have
\begin{equation*}
g^*(p)=\begin{cases}
\sum_{i=1}^{n}p_iY_i,& \text{if } p_iY_i \in [-C,0], i=1, \dots, n,\\
+\infty, & \text{otherwise.}
\end{cases}
\end{equation*}
Consequently, for every $\mu>0$ and $p = (p_1,...,p_n) \in \R^n$, it holds
\begin{equation*}
\text{Prox}_{\mu g^*}(x)=\left(\mathcal{P}_{Y_1[-C,0]}(p_1-\sigma Y_1),\dots,\mathcal{P}_{Y_n[-C,0]}(p_n-\sigma Y_n)\right),
\end{equation*}
where $\mathcal{P}_{Y_i[-C,0]}$ denotes the projection operator on the set $Y_i[-C,0], i=1,...,n$. 

We implemented Proximal AMA for $M_2^k=0$ for all $k \geq 0$ and different choices for the sequence $(M_1^k)_{k \geq 0}$. This resulted in an iterative scheme which reads for all $k \geq 0$:
	\begin{align}
	x^{k+1}&=\argmin_{x \in \R^n}\left\{f(x)-\langle p^k,K x \rangle + \frac{1}{2}\|x-x^{k}\|^2_{M_1^{k}}\right\} =(K+M_1^k)^{-1}(Kp^k+M_1^kx^{k}) \label{eq52}\\
	z^{k+1}&=\text{Prox}_{\frac{1}{c_k}g}\left(Kx^{k+1}-\frac{1}{c^k}p^k\right) =\left(Kx^{k+1}-\frac{1}{c^k}p^k\right)-\frac{1}{c_k}\text{Prox}_{c_k g^*}\left(c_kKx^{k+1}-p^k\right)\label{eq53}\\
	p^{k+1}&=p^k+c_k(-K x^{k+1}+ z^{k+1}).\nonumber
	\end{align}
We would like to emphasize that the AMA method updates the sequence $(z^{k+1})_{k \geq 0}$ also via \eqref{eq53}, while the sequence $(x^{k+1})_{k \geq 0}$, as $M_1^k=0$, is updated via $x^{k+1} = p^k$ for all $k \geq 0$. However, it turned out that the Proximal AMA where $M_1^k=\tau_k K,$ for 
$\tau_k>0$ and all $k \geq 0,$ performs better than the version with $M_1^k=0$ for all $k \geq 0$, which actually corresponds to the AMA method. In this case \eqref{eq52} becomes $x^{k+1}=\frac{1}{1+\tau_k} (p^k+\tau_kx^{k})$ for all $k \geq 0$. 

We used for both algorithms a constant sequence of stepsizes $c_k=2 \cdot \frac{\lambda_{\min}(K)}{\|K\|^2}-10^{-8}$ for all $k \geq 0$. 
The tables below show for $C=1$ and different values of the kernel parameter $\sigma$ that Proximal AMA outperforms AMA in what concerns the time and the number of iterates needed to achieve a certain value for a given fixed misclassification rate (which proved to be the best one among several obtained
by varying $C$ and $\sigma$) and for the RMSE (Root-Mean-Square-Deviation) for the sequence of primal iterates.

\vspace{12pt}
\begin{table}[H]
	\centering
	\begin{tabular}{lll}
		\hline
		Algorithm & misclassification rate at 0.7027 \% & RMSE $\leq 10^{-3}$ \\
		\hline
		Proximal AMA & 8.18s (145) & 23.44s (416)\\
		AMA &8.65s (153) & 26.64s (474)\\
		\hline
	\end{tabular}
	\caption{Performance evaluation of Proximal AMA (with $\tau_k=10$ for all $k \geq 0$) and AMA for the classification problem with $C=1$ and $\sigma=0.2$. The entries refer to the CPU times in seconds and the number of iterations.}
\end{table}

\begin{table}[H]
	\centering
	\begin{tabular}{lll}
		\hline
		Algorithm & misclassification rate at 0.7027 \% & RMSE $\leq 10^{-3}$ \\
		\hline
		Proximal AMA & 141.78s (2448) & 629.52s (10940)\\
		AMA &147.99s (2574) & 652.61s (11368)\\
		\hline
	\end{tabular}
	\caption{Performance evaluation of Proximal AMA (with $\tau_k=102$ for all $k \geq 0$) and AMA for the classification problem with  $C=1$ and $\sigma=0.25$. The entries refer to the CPU times in seconds and the number of iterations.}
\end{table}

\section{Conclusions and further research}\label{sec4}

The Proximal AMA method has the advantage over the classical AMA method that it allows to perform a proximal step for the calculation of $z^{k+1}$ as long as the sequence $M_2^k$ is chosen for all $k \geq 0$ appropriately. In this way one can avoid using in every iteration a minimization subroutine.
It also has more flexibility due to the presence of the smooth and convex functions $h_1$ and $h_2$. In addition, it allows to use proximal terms induced by variable metrics in the calculation of $x^{k+1}$, for all $k \geq 0$, too, which may lead to better performances, as shown in the numerical experiments on support vector machines classification.

In the future, it might be interesting to:

(1) carry out investigations related to the convergence rates for both the iterates and objective function values of Proximal AMA;  as emphasized in \cite{bocs17} for the Proximal ADMM algorithm, 
the use of variable metrics can have a determinant role in this context, as they may lead to dynamic stepsizes which are favourable to an improved convergence behaviour of the algorithm (see  also \cite{bocs15,ch-pck}). 

(2) consider a slight modification of Algorithm \ref{alg:prox-AMA-h}, 
by replacing \eqref{eq:prox-AMA-h-p-Update} with
\begin{equation*}p^{k+1} =  p^k+\theta c_k(b-Ax^{k+1}-Bz^{k+1}),\label{C-var-p-rel}\end{equation*}
where $\theta\in \left(0,\frac{\sqrt{5}+1}{2}\right)$ and to investigate the convergence properties of the resulting scheme; it has been noticed in \cite{fortinglowinski} that the numerical performances of the classical ADMM algorithm for convex optimization problems in the presence of a relaxation parameter  $\theta\in \left(1,\frac{\sqrt{5}+1}{2}\right)$ outperform the ones obtained when $\theta=1$. 

(3) embed the investigations made in this paper in the more general framework of monotone inclusion problems, as it was recently done in \cite{baco17} starting from the Proximal ADMM algorithm.

\end{document}